\documentclass[11pt]{amsart}
\usepackage{amsthm}
\usepackage{geometry}
\usepackage{amsmath,amssymb,mathrsfs}
\usepackage{graphicx}
\newtheorem{thm}{Theorem}[section]
\numberwithin{equation}{section}
\newtheorem{lemma}[thm]{Lemma}
\newtheorem{prop}[thm]{Proposition}

\newtheorem{defin}[thm]{Definition}
\newtheorem{cor}[thm]{Corollary}

\newtheorem{question}[thm]{Question}

\newcommand{\T}{{\mathbb T}}
\newcommand{\N}{{\mathbb N}}
\newcommand{\R}{{\mathbb R}}

\newcommand{\Z}{{\mathbb Z}}
\newcommand{\C}{{\mathbb C}}

\usepackage{enumitem}

\usepackage{multicol}

\title[Random waves on $\T^3$]{Random waves on $\T^3$: nodal area variance and lattice point correlations}
\thanks{The research leading to these results has received funding from the European Research Council under the European Union's Seventh Framework Programme (FP7/2007-2013), ERC grant agreement n$^{\text{o}}$ 335141}
\author{{Jacques Benatar \hspace{1.5cm}    Riccardo W. Maffucci}}

\address{ Department of Mathematics, King's College London, WC2R 2LS, UK.}
\email{jacques.benatar@kcl.ac.uk}
\address{ Department of Mathematics, King's College London, WC2R 2LS, UK.}
\email{riccardo.maffucci@kcl.ac.uk}

\begin{document}
\numberwithin{equation}{section}

\begin{abstract}
We consider the ensemble of random Gaussian Laplace eigenfunctions on $\mathbb{T}^3=\mathbb{R}^3/\mathbb{Z}^3$ (`$3d$ arithmetic random waves'), and study the distribution of their nodal surface area. The expected area is proportional to the square root of the eigenvalue, or `energy', of the eigenfunction. We show that the nodal area variance obeys an asymptotic law. The resulting asymptotic formula is closely related to the angular distribution and correlations of lattice points lying on spheres.
\end{abstract}
\maketitle
\tableofcontents

\section{Introduction}
\subsection{Nodal volume of toral eigenfunctions}
Let $(\mathcal{M},g)$ be a $d$-dimensional compact Riemannian manifold and let $\Delta$ denote the Laplace-Beltrami operator on $\mathcal{M}$. We consider real-valued functions $F:\mathcal{M} \to\mathbb{R}$ satisfying the differential equation
\begin{equation}
\label{laplace}
(\Delta+E)F=0, \quad E>0.
\end{equation}
The nodal set of $F$ is the zero locus
\begin{equation}
\label{nodalsurf}
\{x\in\mathcal{M}: F(x)=0\}.
\end{equation}
It was proven by Cheng \cite{cheng1} that the nodal sets of Laplace eigenfunctions on $\mathcal{M}$  are smooth hypersurfaces, except for a subset of lower dimension. 
We let
$$\mathcal{V}:=\text{Vol}(\{x\in\mathcal{M} : F(x)=0\})$$
denote the {\bf nodal volume} of $F$.

A fundamental conjecture of Yau \cite{yau982}, \cite{yau993} asserts that, for smooth $\mathcal{M}$, the ($(d-1)$- dimensional) nodal volume of a Laplace eigenfunction $F$ with eigenvalue $E$ obeys the sharp bounds 
\begin{equation}
\label{yau}
\sqrt{E}\ll_{\mathcal{M}} \mathcal{V} \ll_{\mathcal{M}}\sqrt{E}.
\end{equation}
This conjecture was established for manifolds $\mathcal{M}$ with a real analytic metric (see Donnelly and Fefferman \cite{donfef}, and Br\"uning and Gromes \cite{brunin}, \cite{brugro}), thus in particular it holds for the torus $\T^d=\R^d/\Z^d$. The lower bound in Yau's conjecture was proven for general smooth $\mathcal{M}$  by Logunov \cite{loguno}.

For $\mathcal{M}=\mathbb{T}^{d}$, the eigenspaces of the Laplacian are related to the theory of lattice points on $(d-1)$-dimensional spheres $\sqrt{m}\mathcal{S}^{d-1}$. Our main focus is the $3$-dimensional torus $\T^3$; in this setting we will call  \eqref{nodalsurf} a {\em nodal surface} and write
\begin{equation}
\label{nodalvolume}
\mathcal{A}:=\text{Vol}(\{x\in\T^3 : F(x)=0\})
\end{equation}
for the {\bf nodal area} of $F$.

The sequence of eigenvalues, or {\em energy levels}, of \eqref{laplace} is
\begin{equation*}
\{E_m=4\pi^2m\}_{m\in V^{(3)}},
\end{equation*}
where
\begin{equation*}
V^{(3)}:=\{m: \ 0<m=a^2+b^2+c^2,\ a,b,c\in\mathbb{Z}\}.
\end{equation*}
For $m\in V^{(3)}$, let
\begin{equation}
\label{lpset}
\mathcal{E}=\mathcal{E}(m,3):=\{\mu\in\mathbb{Z}^3 : \|\mu\|^2=m\}
\end{equation}
be the set of all lattice points on the sphere of radius $\sqrt{m}$. Their cardinality is the number of ways that $m$ may be written as a sum of three squares, and will be denoted by
\begin{equation*}
\mathcal{N}=\mathcal{N}_m:=|\mathcal{E}|
\end{equation*}
(one often writes $r_3(m)$).
Given an eigenvalue $E=4\pi^2 m$, its eigenspace has dimension $\mathcal{N}$, and admits the $L^2$-orthonormal basis
\begin{equation}
\label{basis}
\{e(\mu\cdot x)\}_{\mu\in\mathcal{E}},
\end{equation}
where $e(z):=e^{2 \pi i z}$. All the corresponding eigenfunctions may be written as linear combinations of these exponentials.

\subsection{Arithmetic random waves}
\label{model}
Let us consider the {\em random Gaussian} toral Laplace eigenfunctions (`arithmetic random waves' \cite{orruwi}, \cite{rudwi2}, \cite{krkuwi})
\begin{equation}
\label{arw}
F(x)=\frac{1}{\sqrt{\mathcal{N}}}
\sum_{\mu\in\mathcal{E}}
a_{\mu}
e(\mu\cdot x)
\end{equation}
with eigenvalue $E=4\pi^2m$, where $a_{\mu}$ are complex standard Gaussian random variables (meaning that $\mathbb{E}(a_{\mu})=0$ and $\mathbb{E}(|a_{\mu}|^2)=1$), independent except for the relations $a_{-\mu}=\overline{a_{\mu}}$, which make $F(x)$ real-valued. The arithmetic random wave \eqref{arw} is thus a (centred) Gaussian random field, and the nodal area \eqref{nodalvolume} is a random variable associated to $F$.
\\
The Gaussian field $F$ is {\em stationary}, in the sense that its {\em covariance function} $r$ depends on $x$ only:
\begin{equation}
\label{r}
r(x):=
\mathbb{E}[F(y)\cdot F(x+y)]
=
\frac{1}{\mathcal{N}}\sum_{\mu\in\mathcal{E}} e(\mu\cdot x).
\end{equation}
Every (centred) Gaussian random field is determined by its covariance function via Kolmogorov's Theorem (see e.g. \cite[\S 3.3]{cralea}): the arithmetic random waves are independent of the choice of particular orthonormal basis \eqref{basis}. The normalising factor $1/\sqrt{\mathcal{N}}$ in \eqref{arw}, which clearly has no impact on the nodal area, is chosen so that $F$ is unit-variance (i.e. $r(0)=1$).

\subsection{Prior work on this model}
Let us consider the general setting of arithmetic random waves in dimension $d\geq 1$.
\begin{defin}
\label{lpsetd}
For $d\geq 1$, we denote
\begin{equation}
\mathcal{E}(m,d):=\{\mu\in\mathbb{Z}^d : \|\mu\|^2=m\}
\end{equation}
the set of all lattice points on $\sqrt{m}\mathcal{S}^{d-1}$, and $\mathcal{N}_m^{(d)}$ their number.
\end{defin}
If the index $d$ is omitted it is understood that $d=3$. On $\mathbb{T}^d$, the arithmetic random waves have the expression $F:\mathbb{T}^d\to\mathbb{R}$,
\begin{equation}
\label{arwd}
F(x)=\frac{1}{\sqrt{\mathcal{N}_m^{(d)}}}
\sum_{\mu\in\mathcal{E}(m,d)}
a_{\mu}
e(\mu\cdot x),
\end{equation}
where $m$ is the sum of $d$ integer squares, and $a_{\mu}$ are complex standard Gaussian random variables, independent save for the relations $a_{-\mu}=\overline{a_{\mu}}$, which make $F(x)$ real-valued.

Rudnick and Wigman \cite{rudwi2} investigated the $d-1$-dimensional nodal volume of arithmetic random waves on $\mathbb{T}^d$. They computed the expected value to be, for $d\geq 1$, 
\begin{equation}
\label{rw2008}
\mathbb{E}[\mathcal{V}]=\mathcal{I}_d\sqrt{m},
\end{equation}
where
\begin{equation*} \mathcal{I}_d=\sqrt{\frac{4\pi}{d}}\cdot\frac{\Gamma(\frac{d+1}{2})}{\Gamma(\frac{d}{2})}
\end{equation*}
(see \cite[Proposition 4.1]{rudwi2}). Note that the order of magnitude of \eqref{rw2008} agrees with Yau's conjecture \eqref{yau}. They also gave the following bound for the variance: for $d\geq 2$,
\begin{equation}
\label{rwsqrt}
Var(\mathcal{V})\ll\frac{m}{\sqrt{\mathcal{N}}} \qquad \text{as } \mathcal{N}\to\infty
\end{equation}
(see \cite[Proposition 6.1]{rudwi2}).
As a consequence, the nodal volume concentrates around its mean (see \cite[Section 1.2]{krkuwi}). Rudnick and Wigman (\cite[Section 1]{rudwi2}) conjectured that the stronger bound
\begin{equation}
\label{rwone}
Var(\mathcal{V})\ll\frac{m}{\mathcal{N}}
\end{equation}
should hold.
\\
A deep result of Krishnapur, Kurlberg and Wigman \cite{krkuwi} is
the precise asymptotic behaviour of the variance on $\mathbb{T}^2$ (here the volume is the length of the nodal lines). The energy levels for $d=2$ are the numbers expressible as a sum of two squares 
\begin{equation*}
V^{(2)}=\{m: \ 0<m=a^2+b^2, a,b\in\mathbb{Z}\}.
\end{equation*}
For any subsequence of energies $\{m_i\}_i\subset{V^{(2)}}$ such that the multiplicities $\mathcal{N}_{m_i}\to\infty$, it was shown that (cf. \cite[Theorem 1.1]{krkuwi})
\begin{equation}
\label{length}
Var(\mathcal{V})=c_{m_i}\frac{m_i}{\mathcal{N}_{m_i}^2}(1+o(1)),
\end{equation}
where the positive real numbers $c_{m_i}$ depend on the angular distribution of $\mathcal{E}(m_i,2)$ - the asymptotics for the variance are {\em non-universal} (see \cite[Section 1.2]{krkuwi}).
\\
Also remarkably, the order of magnitude of \eqref{length} is much smaller than expected \eqref{rwone}, as the terms of order $\frac{m}{\mathcal{N}}$ in the asymptotic expression for the nodal length variance cancel perfectly; this effect was called {\em arithmetic Berry cancellation} in \cite{krkuwi}, after ``Berry's cancellation phenomenon" \cite{berry2}.


\subsection{Statement of main results}
The variance of the nodal area has the following precise asymptotics.
\begin{thm}
\label{thethm}
As $m\to\infty$, $m\not\equiv 0,4,7 \pmod 8$, we have
\begin{equation*}
Var(\mathcal{A})=\frac{m}{\mathcal{N}^2}\cdot\left[\frac{32}{375}+O\left(\frac{1}{\mathcal{N}^{1/14-o(1)}}\right)\right].
\end{equation*}
\end{thm}
The $3$-dimensional torus exhibits arithmetic Berry cancellation like the $2$-dimensional torus; the variance has the same order of magnitude as \eqref{length}. See Section \ref{aBc} for more details.
\\
We also remark that, unlike the $2$-dimensional case, the leading order term does not fluctuate: this is because lattice points on spheres are {\em
equidistributed} (see Section \ref{spheres}).
\\
We impose the natural condition $m\not\equiv 0,4,7 \pmod 8$ (cf. \cite[Section 1.3]{ruwiye} and \cite[Section 1.2]{maff3d}) which implies $\mathcal{N}\to\infty$ (see Section \ref{prelim}). Indeed, if $m\equiv 7 \pmod 8$, the set of lattice points $\mathcal{E}(m,3)$ \eqref{lpset} is empty and
\begin{equation*}
\mathcal{E}(4m,3)=\{2\mu : \mu\in\mathcal{E}(m,3)\}
\end{equation*}
(see e.g. \cite[\S 20]{harwri}), hence it suffices to consider energies $m \not \equiv 0 \bmod 4$.

\begin{defin}
\label{correl}
For $\ell\geq 2$, the {\em set of $d$-dimensional $\ell$-th lattice point correlations}, or $\mathbf{\ell}$-{\bf correlations} for short, is
\begin{equation*}
\mathcal{C}_m^{(d)}(\ell):=\Big\{(\mu_1,\dots,\mu_\ell)\in\mathcal{E}(m,d)^{\ell} :
\sum_{i=1}^{\ell}\mu_i=0 
\Big\}.
\end{equation*}
The set of {\bf non-degenerate} $\ell$-correlations is
\begin{equation*}
\mathcal{X}_m^{(d)}(\ell):=\Big\{(\mu_1,\dots,\mu_\ell)\in\mathcal{C}_m^{(d)}(\ell) :
\forall\mathcal{H}\subsetneq\{1,\dots,\ell\}, \sum_{i\in\mathcal{H}}\mu_i\neq 0\Big\}.
\end{equation*}
\end{defin}

We note that, for even $\ell$, the set of $\ell$-correlations is related to the $\ell$-th moment of the covariance function \eqref{r} as follows:
\begin{equation*}
\mathcal{R}(\ell)=\frac{|\mathcal{C}_m^{(d)}(\ell)|}{\mathcal{N}^\ell},
\end{equation*}
where we define
\begin{equation}
\label{rk}
\mathcal{R}(\ell)=\mathcal{R}_m^{(d)}(\ell):=\int_{\mathbb{T}^d}|r^\ell(x)|dx.
\end{equation}
To prove Theorem \ref{thethm}, we shall require the following arithmetic formula.
\begin{prop}[Arithmetic formula]
\label{theprop}
As $m\to\infty$, $m\not\equiv 0,4,7 \pmod 8$, we have
\begin{equation*}
Var(\mathcal{A})=\frac{m}{\mathcal{N}^2}\cdot\left[\frac{32}{375}+O\left(\frac{1}{\mathcal{N}^{1/14-o(1)}}+\frac{|\mathcal{X}(4)|}{\mathcal{N}^2}+\frac{|\mathcal{C}(6)|}{\mathcal{N}^4}\right)\right].
\end{equation*}
\end{prop}
We are naturally led to the following arithmetic problem:
\\
\begin{question} 
How big are the sets $\mathcal{C}_m^{(d)}(\ell)$ and $\mathcal{X}_m^{(d)}(\ell)$ ?
\end{question}
Firstly, it is easy to see that, for every $d$,
\begin{equation}
\label{C2}
|\mathcal{C}_m^{(d)}(2)|={\mathcal{N}_m^{(d)}}.
\end{equation}

The case $d=2$ of this problem was studied in detail by Bombieri and Bourgain \cite{bombou}. We highlight two implications of their results since they are relevant to our own investigations. First, one has the unconditional bound 
\begin{equation*}
|\mathcal{C}_m^{(2)}(6)|=O(\mathcal{N}^{7/2}) \qquad \text{as } \mathcal{N}\to\infty,
\end{equation*}
which is proven via the Szemer\'edi-Trotter Theorem (see \cite[Section 2]{bombou}). The second result concerns the even length correlations, where the number of tuples with pairwise vanishing component vectors is of the order $\mathcal{N}^{ \ell/2}$: for a density $1$ sequence of energy levels $\{m\}$, it follows from \cite[Theorem 17]{bombou} (see also \cite[Lemma 4]{B}), that these tuples make up the majority of the set $\mathcal{C}_m^{(2)}(\ell)$.

Our next two theorems deal with the $3$-dimensional setting. We provide an estimate for the number of correlations and show that, for even $\ell \geq 8$,  the non-degenerate tuples dominate those that cancel pairwise.  

\begin{thm}\label{main4}
Letting $m \rightarrow \infty$, one has the estimate
\begin{equation}\label{fourcorr}
|\mathcal{X}_m^{(3)}(4)|\ll \mathcal{N}^{7/4+ o(1)}. 
\end{equation}
\end{thm}
\begin{thm}\label{main6}
Letting $m \rightarrow \infty$, one has the estimate
\begin{equation}\label{sixcorr}
|\mathcal{C}_m^{(3)}(6)|\ll \mathcal{N}^{11/3+ o(1)}. 
\end{equation}
\end{thm}
\begin{cor}[Long correlations]\label{maincorro}
For any even length $ \ell \geq 8 $ one has the bounds
\begin{align}\label{ellcorr}
\mathcal{N}^{ \ell -3-o(1)} \ll |\mathcal{X}_m(\ell )|\ll \mathcal{N}^{ \ell -7/3+ o(1)  }. 
\end{align}
as $m \rightarrow \infty$, $m\not\equiv 0,4,7 \pmod 8$. The upper bound holds for all $\ell \geq 6$.

\end{cor}

\begin{proof}[Proof of Theorem \ref{thethm}]
Insert the bounds of Theorems \ref{main4} and \ref{main6} into Proposition \ref{theprop}.
\end{proof}
{\bf Remark.} We believe Theorems \ref{main4} and \ref{main6} to be of independent interest. In addition to their application in the proof of Theorem \ref{thethm}, they allow for the study of finer aspects of $\mathcal{A}$. In the companion paper \cite{BCMW}, it is shown by way of Theorems \ref{main4} and \ref{main6}, that in the Wiener chaos expansion of $\mathcal{A}$, only the fourth order chaos component is asymptotically significant: its distribution is asymptotic to the distribution of $\mathcal{A}$.

The lower bound in Corollary \ref{maincorro} indicates that the value distribution of $r(x)$, when averaged over the whole torus, is not Gaussian. For example the eighth moment of $r(x)$ blows up relative to the variance.


\subsection{Outline of the paper}
In the rest of this work, we prove Theorems \ref{main4} and \ref{main6}, and Proposition \ref{theprop}. The proof of Proposition \ref{theprop} begins in Section \ref{seckr} and is concluded in Section \ref{secnsing}, after the necessary preparatory results have been stated. The proof follows the method employed in \cite{krkuwi} for the $2$-dimensional case. The arithmetic random wave $F$ (as in \eqref{arw}) is a Gaussian random field: the variance of the nodal area may be evaluated via the {\em Kac-Rice formulas}, which are discussed in Section \ref{seckr}. To this purpose, it is necessary to understand the (scaled) {\em two-point correlation function} $K_2$ of $F$ (defined in \eqref{K2tdef} and \eqref{K2def}). In Section \ref{seckr}, we express $K_2$ in terms of the conditional Gaussian expectation of the $6\times 6$ vector $(\nabla F(0), \nabla F(x))$ conditioned on $F(0)=0, F(x)=0$; the resulting (scaled) covariance matrix, $\Omega$, depends on the covariance function \eqref{r} and its (first and second order) derivatives.

Next, in Section \ref{secnsing}, we define a small set $S\subset\mathbb{T}^3$ (the {\em singular set}, cf. Definition \ref{S}), where it is possible to bound the contribution of $K_2$ to the variance. We then establish asymptotics for $K_2$ valid outside the set $S$: this computation involves the Taylor expansion of $K_2$ as a $6$-variate function of the matrix $\Omega$ around the identity matrix $I_6$; in fact, we will show that, on $\mathbb{T}^3\setminus S$, $\Omega$ is a small perturbation of $I_6$. The Taylor expansion is carried out in Section \ref{Berry's}, using Berry's method \cite{berry2}. In Section \ref{technical} we perform the technical computations needed to evaluate the leading constant of the nodal area variance; the necessary background on spherical lattice points is covered in Section \ref{pflemma4.6}.

Let us highlight similarities and differences with the $2$-dimensional setting \cite{krkuwi}. Both the leading term and error term in Proposition \ref{theprop} are of arithmetic nature, as in \cite{krkuwi}: the leading term depends on the angular distribution of lattice points on spheres, while the error term depends on the lattice point correlations of Definition \ref{correl}. However, there are marked differences between the $2$- and $3$-dimensional settings; first, as noted above, the nodal area variance obeys an asymptotic law, whereas the nodal {\em length} variance depends on arithmetic properties of the energy.
\\
Second, for the admissibility of the error term, we require a bound for $|\mathcal{X}_m^{(3)}(4)|$ whereas, in the $2$-dimensional setting, 
\begin{equation*}
|\mathcal{X}_m^{(2)}(4)|=0 \qquad \text{for all } m\in{V^{(2)}},
\end{equation*}
which may be seen by noting that two circles intersect in at most two points (Zygmund's trick \cite{zyg}).
The bound for the length four correlations of Theorem \ref{main4} will be established in Section \ref{foursection}.

One must also bound the total number of length six correlations $\mathcal{C}_m(6)$. The proof of Theorem \ref{main6} will be established in Section \ref{sixsection} via a theorem due to Fox, Pach, Sheffer, Suk and Zahl \cite{Zar}. Their result allows one to bound the number of incidences between points and spheres in $\R^3$, thereby playing the role of the Szemer\'edi-Trotter Theorem employed in dimension $2$.

\vspace{0.3cm}
{\bf Notation.} For functions $f$ and $g$ we will use Landau's asymptotic notation $f=O(g)$  or equivalently $f \ll g$ to denote the inequality $f \leq C g$ for some constant $C$. We may add a subscript e.g. $f\ll_t g$ to emphasize the fact that $C$ depends on the parameter $t$. The statement $f\asymp g$ means $g\ll f\ll g$.
\\
The letter $\mu$ will be reserved for elements of $\mathcal{E}_m$ while $\tau$ will denote a member of $\mathcal{E}_m+\mathcal{E}_m$. Generic (deterministic) points $\underline{x} \in \R^3$ will be underlined and we will write $\underline{u} \cdot \underline{x} $ for the Euclidean inner product of two vectors $\underline{u}$ and $\underline{x}$. Finally, we will denote by $S(\underline{x},R) \subset \R^3$ the sphere of radius $R$ centered at $\underline{x}$.

\subsection*{Acknowledgements}
R.M. was funded by a Graduate Teaching Scholarship, Department of Mathematics, King's College London, and worked on this project as part of his PhD thesis, under the supervision of Igor Wigman. The authors wish to thank Igor Wigman for suggesting this collaboration, for insightful remarks and corrections, and for his availability. The authors wish to thank Ze{\'e}v Rudnick and Joshua Zahl for helpful communications and Shagnik Das for pointing out the explicit bound in Theorem \ref{KST}.

\section{Lattice points on spheres and correlations}
\label{pflemma4.6}
\subsection{Preliminary results about the lattice point set}
\label{prelim}
Recall the notation for the lattice point set $\mathcal{E}_m$ and its cardinality $\mathcal{N}_m$. As mentioned in the introduction,
$\mathcal{E}_m$ is non-empty if and only if $m$ is not of the form $4^l(8k+7)$. We work with the assumption $m\not\equiv 0,4,7 \pmod 8$, which is equivalent to the existence of lattice points $(\mu^{(1)},\mu^{(2)},\mu^{(3)}) \in \mathcal{E}_m$ with $\mu^{(1)},\mu^{(2)},\mu^{(3)}$ coprime. In this case, the quantities $\mathcal{N}_m$ and $m$ are related by the estimates
\begin{equation}
\label{totnumlp}
m^{1/2 -o(1)} \ll \mathcal{N}_m\ll m^{1/2 + o(1)}
\end{equation}
(see e.g. \cite[Section 1]{bosaru} or \cite[Section 4]{ruwiye}).

We shall require a bound for the number of lattice points on circles centred in $\Z^3$,
$$\kappa_3(m):= \max_{P} |\mathcal{E}_m \cap P|,$$
where the maximum is taken over all planes $P$. It was shown by V. Jarn\'ik \cite{jarnik} that
\begin{align}\label{taurep}
\kappa_3(m) \ll m^{o(1)}.
\end{align}

Recall Definition \ref{correl} of the set of lattice point $\ell$-correlations $\mathcal{C}=\mathcal{C}_m^{(d)}(\ell)$ and the subset of non-degenerate correlations $\mathcal{X}=\mathcal{X}_m^{(d)}(\ell)$.
In what follows, we may omit the index $d$ when $d=3$, as we will be mostly concerned with the $3$-dimensional setting; we may also suppress the dependency on $m$.
\begin{defin}
\label{correl2}
Denote by $\mathcal{D}=\mathcal{D}_m^{(d)}(\ell)$ the set of {\bf degenerate} correlations so that
\begin{equation*}
\mathcal{C}=\mathcal{D}\dot{\cup}\mathcal{X}.
\end{equation*}
Let $\ell$ be an even positive integer. We will call {\bf symmetric} correlations $\mathcal{D}'$ those that cancel out in pairs. Further, denote by $\mathcal{D}''$ the set of {\bf diagonal} correlations of the form
\begin{equation*}
\{\pm\mu,\dots,\pm\mu\}
\end{equation*}
(with exactly $\ell/2$ plus signs). Note that
\begin{equation*}
\mathcal{D}''\subseteq \mathcal{D}'\subseteq\mathcal{D}.
\end{equation*}
\end{defin}

Let us analyse in detail the set $\mathcal{C}(4)=\mathcal{C}_m^{(3)}(4)$, as several summations range over this set in what follows. Let $d=3$ and $\ell=4$ in Definitions \ref{correl} and \ref{correl2}. Then $\mathcal{D}(4)=\mathcal{D}'(4)$ is the set of quadruples $(\mu_1,\mu_2,\mu_3,\mu_4)$ that cancel out in pairs,
\begin{equation*}
\mu_1=-\mu_2 \text{ and } \mu_3=-\mu_4,
\end{equation*}
and permutations of the indices (i.e., each degenerate correlation is necessarily symmetric when $\ell=4$). The diagonal correlations $\mathcal{D}''\subset\mathcal{D}$ satisfy
\begin{equation*}
\mu_{1}=\mu_{2}=-\mu_{3}=-\mu_{4}
\end{equation*}
for some permutation of the indices.
\\
With $\mathcal{X}(4)$ denoting as usual the set of non-degenerate correlations, a summation over $\mathcal{C}(4)$ may thus be treated by separating it as follows:
\begin{equation}
\label{splitsums}
\sum_{\mathcal{C}(4)}
=
\sum_{\substack{\mu_1=-\mu_2\\\mu_3=-\mu_4}}
+
\sum_{\substack{\mu_1=-\mu_3\\\mu_2=-\mu_4}}
+
\sum_{\substack{\mu_1=-\mu_4\\\mu_2=-\mu_3}}
+
O\left(\sum_{\mathcal{D''}}
+
\sum_{\mathcal{X}(4)}\right)
.
\end{equation}

The proof of Theorem \ref{main4} will rely on a classical estimate regarding the size of the set 
$$I_m(r):=\left\{(\mu_1, \mu_2) \in \mathcal{E}_m^2  \left| \right. \   \mu_1 \cdot \mu_2=r  \right\}=\left\{(\mu_1, \mu_2) \in \mathcal{E}_m^2  \left| \right. \  \| \mu_1 + \mu_2\|^2=2(m+r)  \right\}.$$
In fact, there is an exact formula for $|I_m(r)|$ (see \cite[Section 7]{Pall2}) from which one can deduce the following bound.
\begin{thm}\cite{Pall2}\label{bintern1}
For $|r|<m$ one has that
$$|I_m(r)| \ll \gcd(r,m)^{1/2} m^{o(1)}.$$ 
\end{thm}

Before proceeding to the next lemma we introduce some notation. Given $a \in \N$ write $a_m:=\gcd(a,m)$, yielding the corresponding decomposition $a=a_m a'$. For any interval $J \subset (0,4m)$ we may now introduce the collection $$\mathscr{J}_m(J,a)=\left\{\tau \in \mathcal{E}_m+\mathcal{E}_m |\ \|\tau \|^2 \in J,\  \|\tau \|^2 \equiv 0 \bmod a  \right\}.$$
\begin{lemma}\label{tausizeconversion}
(i) For any $\mathcal{B} \subset \mathcal{E}_m+\mathcal{E}_m$ satisfying the bound $\left|\left\{ ||\tau||^2 \ | \tau \in \mathcal{B} \right\} \right| \leq T$ one has that $|\mathcal{B}| \ll \mathcal{N}^{1 +o(1)} T^{1/2}$.\\
(ii) Given any natural number $a=a_m a'$ and any interval $J \subset (0,4m)$ we have the estimate
$$|\mathscr{J}_m(J,a)| \leq \mathcal{N}^{o(1)}\left(\frac{|J|}{(a_m)^{1/2} a'} + \frac{\mathcal{N}}{(a')^{1/2}}\right).$$
\begin{proof}
(i) Given any $d|m$ the number of $\tau \in \mathcal{B}$ for which $\gcd( \ (||\tau||^2-2m)/2, m)=d$ is at most
\begin{align}\label{estbintern}
\sum'_{| l| <   m/d  }\left|  \left\{ \tau \in \mathcal{B} \left| \  \frac{||\tau||^2-2m}{2}=ld  \right. \right\} \right| \leq
 \max_{\substack{\mathcal{L} \subset (-m/d,m/d) \cap \Z  \\  |\mathcal{L}| \leq T  }    }   \ \sum'_{  l \in \mathcal{L} }  \left|I_m \left(ld \right) \right|,
\end{align}
where the superscript $'$ indicates a summation over integers $l$ for which $\gcd(ld,m)=d$.\\
We first consider divisors in the range $d \geq m/T$. Applying Theorem \ref{bintern1} we gather that the RHS of (\ref{estbintern}) is no greater than
$$ \sum_{l \leq m/d} d^{1/2}m^{o(1)} \ll  \frac{m^{1+{o(1)} }   }{d^{1/2}} \ll \mathcal{N}^{1 +o(1)} T^{1/2}.$$
On the other hand, when $d < m/T$ the  RHS of (\ref{estbintern}) is $O( d^{1/2} T ) = O(\mathcal{N}^{1 +o(1)} T^{1/2})$. Adding the contribution of each divisor $d$ we get the desired estimate.\\
(ii) We repeat the argument given in part (i) and consider for each divisor $d|m$ the vectors $\tau \in \mathscr{J}_m(J,a)$ for which $\gcd( \ (||\tau||^2-2m)/2, m)=d$. In particular we must have that $a_m |d$ and it is not hard to show that $||\tau||^2 \equiv 0 \bmod d a'$ which implies the bound $d \leq 4m/a'$. Setting $J':= \frac{1}{2}J -m $, the inequality (\ref{estbintern}) becomes
\begin{align*}
\sum'_{| l| <   m/d  }\left|  \left\{ \tau \in \mathscr{J}_m(J,a)\left| \  \frac{||\tau||^2-2m}{2}=ld  \right. \right\} \right| & \leq
   \ \sum'_{\substack{l \in J'/d \\  l \equiv -m/d \bmod a'  }  }  \left|I_m \left(ld \right) \right| \\
   & \leq \mathcal{N}^{o(1)}\cdot d^{1/2} \left( \frac{|J'|}{d a'} +1\right)\\
   & \leq \mathcal{N}^{o(1)}\cdot \left( \frac{|J'|}{(a_m)^{1/2} a'} + \left(\frac{m}{a'}\right)^{1/2} \right).
\end{align*}
Noting that $|J'| < |J| $ we add the contribution of each $d$ to conclude the lemma.

\end{proof}
\end{lemma}

\subsection{Equidistribution of lattice points on spheres}
\label{spheres}
Linnik conjectured (and proved under GRH) that the projected lattice points $\mathcal{E}_m/\sqrt{m} \subset \mathcal{S}^2$ become equidistributed as $m\to\infty$, $m\not\equiv 0,4,7 \pmod 8$. This result was proven unconditionally by Duke \cite{duke88}, \cite{dukesp}, and by Golubeva and Fomenko \cite{golfom}. As a consequence, one may approximate a summation over the lattice point set by an integral over the unit sphere:
\begin{lemma}[{\cite[Lemma 8]{pascbo}}]
\label{equidlemma}
Let $g(z)$ be a smooth function on $\mathcal{S}^2$. For $m\to\infty$, $m\not\equiv 0,4,7 \pmod 8$, we have
\begin{equation*}
\frac{1}{\mathcal{N}}\sum_{\mu\in\mathcal{E}}g\left(\frac{\mu}{|\mu|}\right)
=
\int_{z\in\mathcal{S}^2}g(z)\frac{dz}{4\pi}
+
O_g\left(\frac{1}{m^{1/28-o(1)}}\right).
\end{equation*}
\end{lemma}

For each positive integer $k$, define the $k$-th moment of the normalised inner product of two lattice points
\begin{equation}
\label{Bk}
B_k={B_k^{(3)}}(m):=\frac{1}{m^k\mathcal{N}^2}\sum_{\mu_1,\mu_2\in \mathcal{E}}(\mu_1\cdot\mu_2)^k.
\end{equation}
This arithmetic quantity arises naturally in the computation of the leading term of the variance (see Section \ref{technical}). By the equidistribution of lattice points on spheres, we will now show that each $B_k$ has a unique limit as $m\to\infty$.
\begin{lemma}
\label{innerk}
We have:
\begin{equation*}
B_k
=
\begin{cases}
0 & \text { for odd } k;
\\
\frac{1}{3} & \text { for } k=2;
\\
\frac{1}{k+1}
+
O\left(\frac{1}{m^{1/28-o(1)}}\right)
& \text { for even } k\geq 4, \text { as } m\to\infty, \ m\not\equiv 0,4,7 \pmod 8.
\end{cases}
\end{equation*}
In particular,
\begin{equation}
\label{B4}
B_4
=
\frac{1}{5}
+
O\left(\frac{1}{m^{1/28-o(1)}}\right).
\end{equation}
\end{lemma}
\begin{proof}
For odd $k$, the summands of \eqref{Bk} cancel out in pairs, by the symmetry of the set $\mathcal{E}$. For $k=2$, the result was shown in \cite[Lemma 2.3]{rudwi2}. It remains to prove the case of even $k\geq 4$;
we rewrite
\begin{equation*}
B_k
=
\frac{1}{\mathcal{N}^2}\sum_{\mu_1,\mu_2}(\cos(\varphi_{\mu_1,\mu_2}))^k,
\end{equation*}
where $\varphi_{\mu_1,\mu_2}$ is the angle between $\mu_1$ and $\mu_2$. In a moment, we will show that, for all $\mu_1$,
\begin{equation}
\label{equid}
\frac{1}{\mathcal{N}}\sum_{\mu_2}(\cos(\varphi_{\mu_1,\mu_2}))^k
=
\frac{1}{k+1}+O\left(\frac{1}{m^{1/28-o(1)}}\right),
\end{equation}
which implies 
\begin{equation*}
\frac{1}{\mathcal{N}^2}\sum_{\mu_1,\mu_2}(\cos(\varphi_{\mu_1,\mu_2}))^k
=
\frac{1}{\mathcal{N}}\sum_{\mu}\left(\frac{1}{k+1}+O\left(\frac{1}{m^{1/28-o(1)}}\right)\right)
\end{equation*}
hence the result of the present lemma. It remains to show \eqref{equid};
apply Lemma \ref{equidlemma} with $g(\cdot)=\cos^k (\varphi_{\mu_1, \cdot} )$:
\begin{equation}
\label{equidcos}
\frac{1}{\mathcal{N}}\sum_{\mu_2}(\cos(\varphi_{\mu_1,\mu_2}))^k
=
\int_{z\in\mathcal{S}^2}(\cos(\varphi_{\mu_1,z}))^k\frac{dz}{4\pi}
+
O\left(\frac{1}{m^{1/28-o(1)}}\right).
\end{equation}
Write $z=(\sin(\theta)\cos(\psi),\sin(\theta)\sin(\psi),\cos(\theta))$, with spherical coordinates $0\leq\theta\leq\pi$ and $0\leq\psi\leq 2\pi$. As the uniform probability measure $\frac{dz}{4\pi}$ on $\mathcal{S}^2$ is rotation invariant, the integral in \eqref{equidcos} is independent of $\mu_1$, and we may rewrite
\begin{multline}
\label{equidcosbis}
\int_{z\in\mathcal{S}^2}(\cos(\varphi_{\mu_1,z}))^k\frac{dz}{4\pi}
=
\int_{z\in\mathcal{S}^2}(\cos(\varphi_{(0,0,1),z}))^k\frac{dz}{4\pi}
\\=
\frac{1}{4\pi}\int_0^{2\pi}d\psi\int_0^{\pi}\cos^k\theta\sin\theta d\theta
=\frac{1}{k+1}.
\end{multline}
Substituting \eqref{equidcosbis} into \eqref{equidcos} yields \eqref{equid}.
\end{proof}

\section{The proof of Theorems \ref{main4} and \ref{main6}}

\subsection{Length four correlations}
In this subsection, we prove Theorem \ref{main4}.
\label{foursection}
For fixed $\mu_1,\mu_2 \in \mathcal{E}_m$ (with $\mu_1 \neq -\mu_2$), write $\tau=(t_1,t_2,t_3):=-(\mu_1 + \mu_2) $. Clearly any pair of points $\mu_3,\mu_4 \in \mathcal{E}_m$ satisfying 
\begin{align}\label{upairs}
 \mu_3+\mu_4=\tau
\end{align} 
must both lie on the intersection $S(0,m^{1/2} ) \cap S(\tau,m^{1/2} ) $. The resulting intersection is a circle of radius $\rho=(m- \frac{1}{4}\| \tau \|^2)^{1/2}$,  
centered at $\frac{1}{2}\tau$ and is confined to the plane $\left\{ \underline{x} \in \R^3 \left| 2 \tau \cdot \underline{x} = \| \tau \|^2 \right. \right\}$.

As a consequence we may count the number of pairs $(\mu_3,\mu_4)$ satisfying (\ref{upairs}) by estimating the size of the set $\tilde{\mathcal{X}}(\tau)$ consisting of those integer lattice points which lie in the plane
\begin{align}\label{plane}
P: \tau \cdot \underline{x}=0 \qquad \underline{x} \in \R^3
\end{align}
and have norm $2\rho=(4m- \| \tau \|^2)^{1/2}$. A bound for $\mathcal{X}_m(4)$ is then given by

\begin{align}\label{decomp1} 
|\mathcal{X}_m(4)| \ll \sum_{\tau \in \mathcal{E}_m+\mathcal{E}_m}\sum_{ \substack{ \mu_1, \mu_2 \in \mathcal{E}_m\\ \mu_1+\mu_2 = \tau} }  \left( |\tilde{\mathcal{X}}(\tau)|-2\right)_{+},
\end{align}
where the summation takes into account only those pairs $(\mu_1,\mu_2)$ for which $\tilde{\mathcal{X}}(\tau)$ contains at least two non-antipodal points. In the remainder of this subsection we will seek to bound the size of the set $\mathcal{T}:=\{ \tau \in  \mathcal{E}_m+\mathcal{E}_m | \ | \tilde{\mathcal{X}}(\tau)|> 2 \}$.

\begin{prop}\label{4propbis}
With the above notation we have the estimate $|\mathcal{T}| \ll \mathcal{N}^{7/4+o(1)}$.  
\end{prop}
\noindent
Let us first prove Theorem \ref{main4} assuming Proposition \ref{4propbis}.
\begin{proof}[Proof of Theorem \ref{main4}]
Following  (\ref{taurep}) one has the general upper bound $|\tilde{\mathcal{X}}(\tau)| \ll m^{o(1)}$ whenever $\tau \neq 0$. Inserting both this estimate and the bound of Proposition \ref{4propbis} into (\ref{decomp1}) we get (\ref{fourcorr}).
\end{proof}

{\bf The proof of Proposition \ref{4propbis}.}
In order to understand $ \tilde{\mathcal{X}}(\tau)$, we begin with a simple description of $P \cap \Z^3$. Recalling the notation $\tau=(t_1,t_2,t_3)$ let us first set $\gcd(t_1,t_2,t_3)=s$ and write $\tau'=(t_1',t_2',t_3'):=\frac{1}{s} \tau$. Since $\tau'$ is primitive, the lattice $P \cap \Z^3$ has determinant $\| \tau' \|$ (cf. the corollary \cite[page 25]{cassel}) and hence there exist vectors $\underline{A}, \underline{B} \in \Z^3$ with $\underline{A} \times \underline{B}= \tau'$. 
 A generic lattice point in $P$ may be expressed as $k \underline{A}+l \underline{B}$  with $k,l \in \Z$.\\
Let us suppose $\tau \in \mathcal{T}$  and write $n:=4m-\| \tau \|^2$. As $\tau \in \mathcal{T}$, there must be two non-antipodal vectors $\underline{C}=k_1 \underline{A}+l_1 \underline{B}$ and $\underline{D}=k_2 \underline{A}+l_2 \underline{B}$ for which
\begin{equation*}
\|\underline{C} \|^2=\|\underline{D} \|^2=n.
\end{equation*}
Setting $r:=k_1 l_2-k_2 l_1$ we observe that $\underline{C} \times \underline{D}=r (\underline{A} \times \underline{B})=\frac{r}{s} \tau$ and record the inequality
\begin{equation}\label{tausize2}
\| \tau \|^2 \leq \frac{16 s^2 m^2}{r^2}.
\end{equation}
Moreover, noting that $\| \underline{C} \times \underline{D} \|^2=n^2-( \underline{C} \cdot \underline{D})^2$ we obtain the identity

$$\frac{s^2}{r^2} \left(n^2-( \underline{C} \cdot \underline{D})^2  \right) +n=4m.$$
Multiplying both sides of the equation by $4r^2 s^2$  one gets the rearranged expression
$$(2 s^2 n +r^2)^2-(2s^2 \left(\underline{C} \cdot \underline{D} \right)  )^2  =16mr^2 s^2  +r^4$$
and hence
\begin{align}\label{4factor}
\left(2 s^2 n +r^2- 2s^2 \left(\underline{C} \cdot \underline{D} \right) \right) \left(2 s^2 n +r^2+2s^2 \left(\underline{C} \cdot \underline{D} \right) \right)  =16mr^2 s^2  +r^4.
\end{align}
Assuming the equation (\ref{4factor}) has solutions, there must exist a positive $d|16mr^2 s^2  +r^4$ (given by either factor on the LHS of \eqref{4factor}) so that 
\begin{align}\label{nsd}
4 s^2 n +2 r^2=d+ \frac{1}{d} \left(16mr^2 s^2  +r^4\right).
\end{align} 
To count the number of vectors $\tau \in \mathcal{T}$ we will consider equation (\ref{nsd}) in each dyadic interval $r \in [R,2R], s \in [S,2S]$. Here $R$ and $S$ are dyadic powers in the ranges $1 \leq R \leq 2m$ and $1 \leq S \leq m^{1/2}$. 
\begin{lemma}\label{TRSbound}
With $R,S$ as above let $\mathcal{T}(R,S) $ denote the set of $\tau \in \mathcal{T}$ which satisfy equation $(\ref{nsd})$ for some pair of integers $(r,s) \in [R,2R] \times [S,2S]$. Then 
\begin{align}\label{TRS}
|\mathcal{T}(R,S)| \ll \mathcal{N}^{o(1)} \min\left(\frac{m}{S}+\mathcal{N}, \frac{S m^2}{R^2}+\mathcal{N}, \mathcal{N} R^{1/2} S^{1/2} \right).
\end{align}

\begin{proof}
Given $\tau \in \mathcal{T}(R,S)$ with its associated quadruple $(n,r,s,d)$ we recall that 
\begin{align*}
\| \tau \|^2 \equiv 0 \bmod s^2.
\end{align*} 
Setting $s_m:=\gcd(s,m)$ we may write $s=s_m s'$ and put $\nu:=\gcd(s^2, m)$. Clearly $s_m| \nu$ and $\nu | (s_m)^2$ so we are led to a decomposition of the form
$$\nu=s_m \sigma_1, \qquad s_m=\sigma_1 \sigma_2$$
which yields $s^2=\nu (\sigma_2(s')^2)$.
It follows from Lemma \ref{tausizeconversion} part (ii) (with $J=(0,4m)$ and $a=s^2$) and the inequality $(s_m \sigma_1)^{1/2} \sigma_2 \geq (s_m \sigma_1 \sigma_2)^{1/2}=s_m$ that
\begin{align}\label{trsbound}
|\mathcal{T}(R,S)|& \ll  \sum_{\substack{  s_m |m\\ s_m \leq 2S }}  \sum_{s' \asymp S/s_m} \sum_{\sigma_1 \sigma_2 = s_m } \mathcal{N}^{o(1)}\left( \frac{m}{ (s_m \sigma_1)^{1/2} \sigma_2 (s')^2}+\frac{\mathcal{N}}{(\sigma_2)^{1/2} s'} \right) \\ \notag
& \ll  \sum_{\substack{  s_m |m\\ s_m \leq 2S }} \sum_{s' \asymp S/s_m} \mathcal{N}^{o(1)}\left( \frac{m}{s_m (s')^2}+\frac{\mathcal{N}}{s'} \right)  \\ \notag
&\ll \mathcal{N}^{o(1)} \left(\frac{m}{S}+ \mathcal{N} \right),
\end{align}
yielding the first inequality in (\ref{TRS}). In light of \eqref{tausize2} we may reuse the estimates given in (\ref{trsbound}), this time applying Lemma \ref{tausizeconversion} part (ii) with the interval $J=(0, 16s^2 m^2/R^2)$. The bound $|\mathcal{T}(R,S)| \ll \mathcal{N}^{o(1)} (Sm/R^2+\mathcal{N})$ follows readily.\\ 
A brief inspection of (\ref{nsd}) reveals that for each choice of $(r,s) \in [R,2R] \times [S,2S] $ and each choice of divisor $d | 16mr^2 s^2  +r^4$, the value of $n$ is uniquely determined. In this manner we get $O(\mathcal{N}^{o(1)} RS)$ possible values of $n$ and hence the final estimate in (\ref{TRS}) follows from an application of Lemma \ref{tausizeconversion} part (i). 

\end{proof}
\end{lemma}
To conclude the proof of Proposition \ref{4propbis}, 
we note that $|\mathcal{T}|\leq \sum_{R,S} |\mathcal{T}(R,S)|$ and apply the estimates of Lemma \ref{TRSbound} to get
\begin{align*}
|\mathcal{T}| &\ll \sum_{   \substack{R\leq 2 m,S \leq m^{1/2} \\ \text{dyadic} }    } \mathcal{N}^{o(1)}   \min\left( \mathcal{N} R^{1/2} S^{1/2} ,\frac{Sm^2}{R^2} ,\frac{m}{S} \right) + \mathcal{N}^{1+o(1)}.
\end{align*}
For fixed $S$, the largest possible value of $\min( \mathcal{N} R^{1/2} S^{1/2} ,Sm^2/R^2)$ occurs when $R \asymp S^{1/5}m^{4/5}/\mathcal{N}^{2/5}$. Recalling the relation between $m$ and $\mathcal{N}$ \eqref{totnumlp},
\begin{equation*}
\min( \mathcal{N} R^{1/2} S^{1/2} ,Sm^2/R^2)\ll\mathcal{N}^{8/5+o(1)} S^{3/5}.
\end{equation*}
It follows that
\begin{align*}
|\mathcal{T}| \ll \sum_{   \substack{S \leq m^{1/2}  }  } \mathcal{N}^{o(1)}   \min\left(\mathcal{N}^{8/5} S^{3/5} ,\frac{m}{S} \right) + \mathcal{N}^{1+o(1)}  \ll \mathcal{N}^{7/4 + o(1)}.
\end{align*}

\subsection{Length six correlations} \label{sixsection}
In this subsection, we prove Theorem \ref{main6}. The key ingredient is the incidence bound \cite[Theorem 6.4]{Zar}, which we state below in a simplified form.
\\
Given a collection of points $\mathcal{P}$ and a collection of varieties $\mathcal{V}$, we define
\begin{equation*}
I(\mathcal{P},\mathcal{V}):=|\{(p,V) \in  \mathcal{P} \times \mathcal{V}\ | \ p \in V \}|
\end{equation*}
to be the number of incidences between $\mathcal{P}$ and $\mathcal{V}$.\\
We will use the standard notation $K_{s,t}$ for complete bipartite graphs. Given graphs $G$ and $H$, we say $G$ is $H$-free if it does not contain an induced subgraph isomorphic to $H$.

\begin{thm}\cite{Zar}\label{genipv}
Let $\mathcal{P} \subset \R^3$ be a set of $k$ points and $\mathcal{V} $ a collection of $n$ varieties of bounded degree in $\R^3$. Assuming the incidence graph of $\mathcal{P}\times \mathcal{V} $ is $K_{s,t}$-free there exists, for each $\varepsilon >0$, a positive constant $c=c(\varepsilon)$ so that
\begin{align}\label{gentinc}
I(\mathcal{P},\mathcal{V} )\leq  st c\left(  k^{ \frac{2s}{3s-1}+\varepsilon  } n^{\frac{3(s-1) }{3s-1}  } + (k+n) \right).
\end{align}
\end{thm}

{\bf Remark.} The inequality (\ref{gentinc}) gives a polynomial dependence in $t$ which will be crucial to the argument in this subsection. Although not explicitly stated in the above form  one can follow the proofs given in \cite[Theorems 4.3 and 6.4]{Zar} and keep track of all the constants involved. We will carry out these straightforward modifications in Appendix \ref{incidencesection}.

To prove Theorem \ref{main6}, we will apply Theorem \ref{genipv} with the set of points $\mathcal{P}= \mathcal{E}+\mathcal{E} $ and varieties $\mathcal{S}=\left\{S(\underline{A}, m^{1/2}) \ |\  \underline{A} \in \mathcal{E}+ \mathcal{E}+\mathcal{E} \right\}$. For fixed $\varepsilon>0$ and $m$ sufficiently large we set $s=2$ and $t=\mathcal{N}^{\varepsilon}$ and observe that, by \eqref{taurep},   
the incidence graph of $\mathcal{P} \times \mathcal{S}$ is $K_{s,t}$-free.\\ 
The remainder of the argument is carried out as in  \cite[Section 2]{bombou} with Theorem \ref{genipv} replacing the Szemer\'edi-Trotter Theorem. For any dyadic power $D \geq 1$ denote by $\mathcal{S}(D)$ the collection of spheres $S=S(\underline{A}, m^{1/2} ) \in \mathcal{S}$ for which $|S \cap \mathcal{P}| \asymp D$. Recalling (\ref{taurep}) we gather that 
\begin{align}\label{Darrangement}
|\mathcal{C}_m(6) | \ll_{\varepsilon} \mathcal{N}^{ \varepsilon }\sum_{\substack{  D \leq \mathcal{N}    \\  \text{dyadic} }} D^2  |\mathcal{S}(D)|.
\end{align}

\begin{lemma}
For $D \leq \mathcal{N}$ we have the estimates 
$$ (i)\  D |\mathcal{S}(D)| \ll \mathcal{N}^3, \qquad (ii) \  D^{5/2} |\mathcal{S}(D)| \ll \mathcal{N}^4.$$
\begin{proof}
(i) For each  $\tau \in \mathcal{P}= \mathcal{E}+\mathcal{E}$ denote by $\mathcal{S}_{\tau}(D)$ the collection of spheres in $\mathcal{S}(D)$ which are incident to $\tau$ . Then we have the trivial bound
$$D |\mathcal{S}(D)| \leq I(\mathcal{P},\mathcal{S}(D)  ) \leq \sum_{\tau \in \mathcal{P} } |\mathcal{S}_{\tau}(D)|\leq \mathcal{N}^3.$$
(ii) We first note the inequality $|\mathcal{S}(D)|\leq |\mathcal{S}(D)|^{3/5} |\mathcal{P}|^{4/5}$ which follows easily from the rearranged statement $|\mathcal{S}(D)|\leq \mathcal{N}^4$. Applying Theorem \ref{genipv} we get the bound 
\begin{align}\label{incidence}
D |\mathcal{S}(D)| \ll I(\mathcal{P},\mathcal{S}(D)  )\ll_{\varepsilon}  |\mathcal{S}(D)|^{3/5} |\mathcal{P}|^{4/5 +\varepsilon}+ |\mathcal{S}(D)| +|\mathcal{P}| \leq |\mathcal{S}(D)|^{3/5} |\mathcal{P}|^{4/5 +\varepsilon} +|\mathcal{P}| .
\end{align}
When the first term on the RHS of (\ref{incidence}) dominates one finds that $D |\mathcal{S}(D)|\ll |\mathcal{S}(D)|^{3/5} |\mathcal{N}|^{8/5} $ which gives $D^{5/2} |\mathcal{S}(D)| \ll \mathcal{N}^4$. When the second term on the RHS dominates we get $D |\mathcal{S}(D)| \ll \mathcal{N}^2$ so that
$$D^{5/2} |\mathcal{S}(D)|\leq D |\mathcal{S}(D)| \mathcal{N}^{3/2} \ll \mathcal{N}^{7/2}.$$
\end{proof}
\end{lemma} 
Combining the estimates of the lemma with (\ref{Darrangement}) we get 
$$|\mathcal{C}_m(6) | \ll_{\varepsilon} \mathcal{N}^{\varepsilon } \sum_{D \leq \mathcal{N}} D\cdot  \min \left(\mathcal{N}^3, \frac{\mathcal{N}^4 }{D^{3/2}} \right) \ll_{\varepsilon} \mathcal{N}^{11/3 + \varepsilon},$$
which completes the proof of Theorem \ref{main6}.

\subsection{Long correlations }
In this section we will prove Corollary \ref{maincorro} via an analytic argument. We introduce the function $f(\underline{ \alpha}):= \sum_{\mu \in \mathcal{E}_m} e(\mu \cdot \underline{ \alpha})$ and observe that, by orthogonality,
\begin{align}\label{circleidea}
|\mathcal{C}_m( \ell)|=\int_{[0,1]^3} f(\underline{ \alpha})^{ \ell} \ d \underline{ \alpha}.
\end{align}

\subsubsection{An upper bound for $|\mathcal{C}_m( \ell)|$} 

Let $\ell \geq 6$ and observe that one has the trivial bound $|f(\underline{\alpha})| \leq \mathcal{N}$.  Combining the estimate (\ref{sixcorr}) with (\ref{circleidea}) it follows that 
$$|\mathcal{C}_m( \ell)| \leq \mathcal{N}^{  \ell -6}  \int |f(\underline{ \alpha})|^{6}  \ d \underline{ \alpha} 
\ll_{\varepsilon}   \mathcal{N}^{ \ell -7/3+\varepsilon}.$$
{\bf Remark.} To conclude the discussion of the upper bounds we record the straightforward estimate
$$|\mathcal{C}_m(5)|\leq \int_{[0,1]^3} |f(\underline{ \alpha})|^{5} \ d \underline{ \alpha} \leq  
\left( \int_{[0,1]^3} |f(\underline{ \alpha})|^{4} \ d \underline{ \alpha}\right) ^{1/2}  
 \left(\int_{[0,1]^3} |f(\underline{ \alpha})|^{6} \ d \underline{ \alpha} \right)^{1/2} \ll_{\varepsilon} \mathcal{N}^{17/6 + \epsilon} $$ 
and note that $|\mathcal{C}_m(2)|= \mathcal{N}$ while $|\mathcal{C}_m(3)| \ll \mathcal{N}^{ 1+ o(1)}$ (as a consequence of (\ref{taurep})).

\subsubsection{A lower bound for $|\mathcal{X}_m( \ell)|$} 
Let $\ell \geq 8$ be even and recall the notation $\mathcal{D}( \ell)$ and $\mathcal{D}'( \ell)$ for the set of degenerate and symmetric tuples respectively. Observe that the degenerate tuples in $\mathcal{C}_m( \ell)$ number at most
$$|\mathcal{D}( \ell)|
=
|\mathcal{D}'( \ell)|+\sum_{\substack{    \substack{ 2 \leq j_1 \leq ... \leq j_k \leq \ell-2 \\ j_1+...+j_k= \ell}  \\ 3 \leq j_k  }     } \ \prod_{i=1}^{k}|\mathcal{C}_m( j_i)| \ll \mathcal{N}^{\ell /2} +\mathcal{N}^{(\sum_{i \leq k} j_i )-1-7/3  + o(1)} \ll \mathcal{N}^{\ell - 10/3 + o(1)} $$ 
with the largest contribution coming from the multi-index $j_1=2, j_2= \ell -2$. As a result it will suffice to prove the asserted lower bound in (\ref{ellcorr}) for $|\mathcal{C}_m( \ell)|$.

Consider the set $A:=\left\{ \alpha \in [0,1]^3 | \ | f(\underline{ \alpha})| \geq \mathcal{N}/2 \right\}$. Since $f(\underline{0})=\mathcal{N} $ and $f$ has partial derivatives of size at most $m^{1/2 } \mathcal{N} \ll_{\varepsilon} \mathcal{N}^{2+ \varepsilon}  $, we gather that
\begin{equation}\label{originmass}
|f(\underline{ \alpha})-\mathcal{N}|=|f(\underline{ \alpha})-f(\underline{0})| \leq ||\nabla f||_{\infty} \cdot ||\underline{\alpha}||\ll_{\varepsilon} \mathcal{N}^{2+ \varepsilon} ||\underline{\alpha}||.
\end{equation} 
It follows that $|f(\underline{ \alpha})|\geq \mathcal{N}/2$ whenever $||\underline{\alpha}||\ll_{\varepsilon} \mathcal{N}^{-1-\varepsilon}$ and hence the Lebesgue measure of $A$ is bounded from below by $\lambda(A)\gg_\varepsilon \mathcal{N}^{-3-\varepsilon}$. Inserting this information into (\ref{circleidea}) we find the desired estimate
$$|\mathcal{C}_m( \ell)| \geq \int_{A} f(\underline{ \alpha})^{  \ell} \ d \underline{ \alpha}\geq \lambda(A) (\mathcal{N}/2)^{\ell}  \gg_{\varepsilon}  \mathcal{N}^{\ell-3-\varepsilon}.$$

\noindent {\bf Remark.} When $d \geq 5$ one can repeat the preceding argument to show that the set of non-degenerate tuples $\mathcal{X}^{(d)}_m(4)$ is much {\em larger} than $\mathcal{D}^{(d)}_m(4)$, as opposed to what happens in dimensions $2$ and $3$. Indeed, for $\mathcal{N}=\mathcal{N}_m^{(d)}$ we set
$$f(\underline{ \alpha}):= \sum_{\mu \in \mathcal{E}^{(d)}_m} e(\mu \cdot \underline{ \alpha}), \qquad A^{(d)}:=\left\{ \alpha \in [0,1]^d \ \big| \ | f(\underline{ \alpha})| \geq \mathcal{N}/2 \right\}$$
and recall the estimates (\cite[Theorem 20.2]{iwakow})
$$m^{d/2-1} \ll_d \mathcal{N}^{(d)}_m \ll_d m^{d/2-1}.$$
Proceeding as in \eqref{originmass} one finds that 
$$|f(\underline{ \alpha})-\mathcal{N}| \ll m^{1/2} \mathcal{N} \cdot ||\underline{\alpha}||\ll \mathcal{N}^{(d-1)/(d-2)} \cdot ||\underline{\alpha}||.$$
As a result $\lambda_d(A^{(d)})\gg \mathcal{N}^{-d/(d-2)}$ (where $\lambda_d$ denotes the $d$-dimensional Lebesgue measure) and hence
$$|\mathcal{C}_m^{(d)}(4)| \geq \int_{A^{(d)}} f(\underline{ \alpha})^{ 4} \ d \underline{ \alpha} \gg \mathcal{N}^{4-d/(d-2)}\gg  \mathcal{N}^{7/3}.$$

\section{Kac-Rice formulas}
\label{seckr}
\noindent
For a smooth random field, the moments of the geometric measure of the nodal set are given by the \textbf{Kac-Rice formulas} (see 
\cite{azawsc}, Theorems 6.8 and 6.9). The arithmetic random wave $F$ \eqref{arw} is a Gaussian field; for each $x\in\mathbb{T}^3$, let $\phi_{F(x)}$ be the probability density function of the (standard Gaussian) random variable $F(x)$, and $\phi_{F(x),F(y)}$ the joint density of the random vector $(F(x),F(y))$. We define the \textbf{zero density} function (also called {\em first intensity}) $K_1: \mathbb{T}^3\to\mathbb{R}$ and \textbf{2-point correlation} function (also called {\em second intensity}) $\tilde{K_2}: \mathbb{T}^3\times\mathbb{T}^3\to\mathbb{R}$ of $F$ as the conditional Gaussian expectations
\begin{equation*}
K_1
=\phi_{F(y)}(0)\cdot\mathbb{E}[\|\nabla F(y)\|\ \big| \ F(y)=0]
\end{equation*}
and
\begin{equation}
\label{K2tdef}
\tilde{K_2}(x)
=\phi_{F(y),F(x+y)}(0,0)\cdot\mathbb{E}[\|\nabla F(y)\|\cdot\|\nabla F(x+y)\|\ \big| \ F(y)=F(x+y)=0],
\end{equation}
the latter defined for $x\neq 0$. The functions $K_1$ and $\tilde{K_2}$ do not depend on $y$, since $F$ is stationary. The Kac-Rice formulas for the first and second moments of the nodal area are
\begin{equation}
\label{kacrice1}
\mathbb{E}(\mathcal{A})=
\int_{\mathbb{T}^3}K_1dx=K_1
\end{equation}
and
\begin{equation}
\label{kacrice2}
\mathbb{E}(\mathcal{A}^2)=
\int_{\mathbb{T}^3}\tilde{K_2}(x)dx.
\end{equation}
As mentioned in the introduction, the expected nodal area was computed by Rudnick and Wigman to be (\eqref{rw2008} with $d=3$)
\begin{equation}
\label{expect}
\mathbb{E}[\mathcal{A}]=\frac{4}{\sqrt{3}}\sqrt{m}.
\end{equation}

It is more convenient to work with a scaled version of the second intensity,
\begin{equation}
\label{K2def}
K_2(x):=\frac{3}{E}\tilde{K_2}(x),
\end{equation}
where we recall that
$E=E_m=4\pi^2m$.
Applying the Kac-Rice formulas, we obtain the following precise expression for the variance of the nodal area.
\begin{prop}
\label{precisevar}
\begin{equation}
\label{precisevarexpression}
Var(\mathcal{A})=\frac{E}{3}\int_{\mathbb{T}^3}\left(K_2(x)-\frac{4}{\pi^2}\right)dx.
\end{equation}
\end{prop}
\begin{proof}
By \eqref{kacrice2} and \eqref{expect},
\begin{multline*}
\mathbb{E}[\mathcal{A}^2]-(\mathbb{E}[\mathcal{A}])^2
=
\int_{\mathbb{T}^3}\tilde{K_2}(x)dx-\frac{16}{3}m
=
\int_{\mathbb{T}^3}\left(\tilde{K_2}(x)-\frac{16}{3}m\right)dx
\\=
\frac{E}{3}\int_{\mathbb{T}^3}\left(K_2(x)-\frac{3}{E}\frac{16}{3}m\right)dx
=
\frac{E}{3}\int_{\mathbb{T}^3}\left(K_2(x)-\frac{4}{\pi^2}\right)dx.
\end{multline*}
\end{proof}

By the above arguments, to understand the nodal area variance of the arithmetic random wave $F$, we need to study the (scaled) two-point function $K_2$; let us begin by introducing the necessary notation. Recall the covariance function $r$ of $F$ is given by \eqref{r}. Let
\begin{equation}
\label{D}
D(x):=\nabla r(x)=\frac{2\pi i}{\mathcal{N}}
\sum_{\mu\in\mathcal{E}} e(\mu\cdot x)\cdot\mu,
\end{equation}
where for $j=1,2,3$ we have computed the partial derivatives
\begin{equation*}
D_j(x)=\frac{\partial r}{\partial x_j}(x)=\frac{2\pi i}{\mathcal{N}}\sum_{\mu\in\mathcal{E}} e(\mu\cdot x)\mu^{(j)},
\qquad\qquad
\mu=(\mu^{(1)},\mu^{(2)},\mu^{(3)}).
\end{equation*}
Further, denote
\begin{equation}
\label{H}
H(x):=
-\frac{4\pi^2}{\mathcal{N}}
\sum_{\mu\in\mathcal{E}} e(\mu\cdot x)\cdot\mu^t\mu
\end{equation}
the Hessian $3\times 3$ matrix of $r$, where for $j,k=1,2,3$,
\begin{equation*}
H_{jk}(x)
:=
\frac{\partial^2 r}{\partial x_j\partial x_k}(x)=-\frac{4\pi^2}{\mathcal{N}}\sum_{\mu\in\mathcal{E}} e(\mu\cdot x)\mu^{(j)}\mu^{(k)}
.
\end{equation*}
The $n\times n$ identity matrix will be  denoted $I_n$.

\begin{prop}
\label{explicit}
The scaled two-point correlation function may be expressed as
\begin{equation}
\label{expl}
K_2(x)=
\frac{1}{2\pi\sqrt{1-r^2(x)}}
\cdot
\mathbb{E}[\|w_1\|\cdot\|w_2\|],
\end{equation}
where $w_1,w_2$ are three-dimensional random vectors with Gaussian distribution $(w_1,w_2)\sim\mathcal{N}(0,\Omega(x))$; their covariance matrix is given by
\begin{equation}
\label{omega}
\Omega=I_6+\begin{pmatrix}
X & Y \\ Y & X
\end{pmatrix},
\end{equation}
the $3\times 3$ matrices $X$ and $Y$ being defined as
\begin{equation}
\label{X}
X(x)=-\frac{1}{1-r^2}\frac{3}{E}\cdot D^tD
\end{equation}
and
\begin{equation}
\label{Y}
Y(x)=-\frac{3}{E}\cdot
\left(
H
+\frac{r}{1-r^2}\cdot
D^tD
\right).
\end{equation}
\end{prop}
\begin{proof}
As $F$ is stationary, \eqref{K2tdef} may be rewritten as
\begin{equation}
\label{K2tilde0}
\tilde{K_2}(x)=\phi_{F(0),F(x)}(0,0)\cdot\mathbb{E}[\|\nabla F(0)\|\cdot\|\nabla F(x)\|\ \big|\  F(0)=F(x)=0].
\end{equation}
Since the covariance matrix of $(F(0),F(x))$ is
\begin{equation}
\label{A}
A(x)=
\begin{pmatrix}
1 & r(x) \\r(x) & 1
\end{pmatrix},
\qquad
\end{equation}
the joint Gaussian density equals
\begin{equation}
\label{K2tilde1}
\phi_{F(0),F(x)}(0,0)
=\frac{1}{2\pi\sqrt{1-r^2(x)}}.
\end{equation}
By \cite[Lemma 5.1]{rudwi2}, the covariance matrix of the eight-dimensional Gaussian vector
\begin{equation*}
\left(
F(0),F(x),
\nabla F(0),
\nabla F(x)
\right)
\end{equation*}
is the block matrix
$
\Sigma(x)=\begin{pmatrix}
A & B \\ B^t & C
\end{pmatrix},
$ 
with $A$ as in \eqref{A},
\begin{equation*}
B(x)=
\begin{pmatrix}
0_{1\times 3} & D_{1\times 3}(x)\\-D_{1\times 3}(x) & 0_{1\times 3}
\end{pmatrix}
\qquad\text{and}\qquad
C=
\begin{pmatrix}
\frac{E}{3}I_3 & -H(x)
\\
-H(x) & \frac{E}{3}I_3
\end{pmatrix}.
\end{equation*}
By \cite[Section 2.3]{orruwi} (also see \cite[Proposition 2.4 (1)]{rudwi2}), there are only finitely many $x\in{\T}^3$ such that $r(x)=\pm 1$. Therefore, for almost all $x\in{\T}^3$, the covariance matrix $A(x)$ is nonsingular. In view of \cite[Proposition 1.2]{azawsc} (see also the hypotheses of \cite[Theorem 6.9]{azawsc}), the covariance matrix of
$
\left(
\nabla F(0),
\nabla F(x)
\right)
$
conditioned on $F(0)=0,F(x)=0$ is $\tilde{\Omega}(x):=C-B^tA^{-1}B$. We then have
\begin{equation}
\label{K2tilde2}
\mathbb{E}[\|\nabla F(0)\|\cdot\|\nabla F(x)\| \ \big| \ F(0)=F(x)=0]
=
\mathbb{E}[\|v_1\|\cdot\|v_2\|],
\end{equation}
where $v_1, v_2$ are three-dimensional random vectors with $(v_1,v_2)\sim N(0,\tilde{\Omega})$.
Inserting \eqref{K2tilde1} and \eqref{K2tilde2} into \eqref{K2tilde0} we obtain
\begin{equation}
\tilde{K_2}(x)=\frac{1}{2\pi\sqrt{1-r^2(x)}}\cdot\mathbb{E}[\|v_1\|\cdot\|v_2\|],
\qquad
(v_1,v_2)\sim N(0,\tilde{\Omega}).
\end{equation}
Lastly, to prove the expression \eqref{expl} for the scaled two-point function, we rescale the random vectors
\begin{equation*}
v_i=:\sqrt{\frac{E}{3}}w_i \qquad i=1,2, 
\end{equation*}
and the matrix
\begin{equation*}
\tilde{\Omega}=:\frac{E}{3}\Omega;
\end{equation*}
then $\Omega$ is given by \eqref{omega}, with $X,Y$ as in \eqref{X} and \eqref{Y}.
\end{proof}
In the proof of the latter proposition, we saw that the distribution of $(w_1,w_2)$ is non-degenerate (i.e., the matrix $\Omega(x)$ is nonsingular) for almost all $x$. Also note that \eqref{omega} expresses $\Omega(x)$ as a perturbation of the identity matrix, in the sense that the entries of $X(x),Y(x)$ are small for `typical' $x\in\mathbb{T}^3$. 

\section{Proof of Proposition \ref{theprop}}
\label{secnsing}
\subsection{The contribution of the singular set}
We will define a small subset of the torus, called the {\em singular set} $S$: outside of $S$, we will eventually establish precise asymptotics for the two-point correlation function $K_2$ (recall \eqref{K2def} and \eqref{K2tdef}). The goal of the present subsection is to bound $K_2$ on $S$, and also to control the measure of $S$. The definitions and results of the present section are borrowed from \cite{orruwi}, \cite{rudwi2} and \cite{krkuwi}. Recall the notation $\mathcal{E}$ for the set of all lattice points on the sphere of radius $\sqrt{m}$. 
\begin{defin}
\label{singpoint}
We call the point $x\in\mathbb{T}^3$ \textbf{positive singular} (resp. \textbf{negative singular}) if there exists a subset $\mathcal{E}_x\subseteq\mathcal{E}$ with density $\frac{|\mathcal{E}_x|}{|\mathcal{E}|}>\frac{11}{12}$ such that $\cos(2\pi(\mu\cdot x))>\frac{3}{4}$ (resp. $\cos(2\pi(\mu\cdot x))<-\frac{3}{4}$) for all $\mu\in\mathcal{E}_x$.
\end{defin}
For instance, the origin $(0,0,0)$ is a positive singular point. Take $q\asymp \sqrt{m}$ and partition the torus into $q^3$ cubes, each centred at $a/q$, $a\in\mathbb{Z}^3$, of side length $1/q$. Note that the cubes have disjoint interiors.
\begin{defin}
\label{singcube}
We call the cube $Q\subset\mathbb{T}^3$ {\em positive singular} (resp. {\em negative singular}) if it contains a positive (resp. negative) singular point.
\end{defin}
\begin{defin}
\label{S}
The {\bf singular set} $S$ is the union of all positive and negative singular cubes.
\end{defin}
The main result of the present subsection is the bound for the integral of $K_2$ on $S$, for which we shall need two lemmas. The covariance function $r$ of the arithmetic random wave $F$ satisfies $|r(x)|\leq 1$. The following lemma shows that, on $S$, $r$ is bounded away from $0$.
\begin{lemma}
[{\cite[Lemmas 6.4 and 6.5]{orruwi}}]
\label{rlower}
$\ $
\begin{enumerate}
\item
For all positive (resp. negative) singular cubes $Q$, there exists a subset $\mathcal{E}_{Q}\subseteq\mathcal{E}$ with density $\frac{|\mathcal{E}_x|}{|\mathcal{E}|}>\frac{11}{12}$ such that for all $y\in Q$ and for all $\mu\in\mathcal{E}_{Q}$, we have 
\begin{equation*}
\cos(2\pi(\mu\cdot y))>\frac{1}{2}
\end{equation*}
(resp. $\cos(2\pi(\mu\cdot y))<-1/2$).
\item
For all $y\in S$:
\begin{equation*}
|r(y)|>\frac{3}{8}.
\end{equation*}
\end{enumerate}
\end{lemma}
\noindent
Recall the definitions \eqref{X} and \eqref{Y} for the matrices $X(x)$ and $Y(x)$.
\begin{lemma}[cf. {\cite[Lemma 3.2]{krkuwi}}]
\label{O(1)}
We have uniformly (entry-wise)
\begin{equation}
\label{XYO1}
X(x)=O(1), \quad Y(x)=O(1).
\end{equation}
One immediate consequence is
\begin{equation}
\label{K2bdd}
K_2(x)\ll \frac{1}{\sqrt{1-r^2(x)}}.
\end{equation}
\end{lemma}
\noindent
Recall the notation $\mathcal{R}(\ell)$ \eqref{rk} for the $\ell$-th moment of the covariance function $r$.
\begin{prop}[cf. {\cite[Section 6.3]{orruwi}} and {\cite[Lemma 4.4]{krkuwi}}]
\label{singset} 
$\ $
\begin{enumerate}
\item
The contribution of the singular set to \eqref{precisevarexpression} has the following bound:
\begin{equation*}
\int_{S}|K_2(x)|dx\ll meas(S).
\end{equation*}
\item
For all integers $\ell\geq 0$:
\begin{equation*}
meas(S)\ll \mathcal{R}(\ell).
\end{equation*}
\end{enumerate}
\end{prop}
\noindent
We end this subsection with a property of the covariance function outside the singular set.
\begin{lemma}
[{\cite[Lemma 6.5]{orruwi}}]
\label{rupper}
For all $x\notin S$, $|r(x)|$ is bounded away from $1$:
\begin{equation*}
|r(x)|\leq 1-\frac{1}{48}.
\end{equation*}
\end{lemma}
Thanks to the lemma, on the non-singular set $\mathbb{T}^3\setminus S$ we have the following approximations:
\begin{equation}
\label{tay1}
\frac{1}{\sqrt{1-r^2}}=1+\frac{1}{2}r^2+\frac{3}{8}r^4+O(r^6)
\end{equation}
and
\begin{equation}
\label{tay2}
\frac{1}{1-r^2}=1+r^2+O(r^4).
\end{equation}

\subsection{Asymptotics for $K_2$ on the non-singular set}
\begin{lemma}
\label{lemma5.1}
Let $(w_1,w_2)\sim N(0,\Omega)$, $\Omega=I_6+
\begin{pmatrix}
X & Y \\ Y & X
\end{pmatrix}$, with $rank(X)=1$. Then:
\begin{multline*}
\mathbb{E}[\|w_1\|\cdot\|w_2\|]
=
\frac{8}{\pi}
\cdot
\left[
1+
\frac{tr(X)}{3}+
\frac{tr(Y^2)}{18}
-
\frac{tr(XY^2)}{45}
-
\frac{tr(X^2)}{45}
\right.
\\
\left.
+
\frac{tr(Y^4)}{900}
+
\frac{tr(Y^2)^2}{1800}
-
\frac{tr(X)tr(Y^2)}{90}
\right]
+
O(tr(X^3)+tr(Y^6))
.
\end{multline*}
\end{lemma}
\noindent
The proof of Lemma \ref{lemma5.1} is quite lengthy and takes up the whole of Appendix \ref{Berry's}. Assuming it, we arrive at the asymptotics for $K_2$ on $\mathbb{T}^3\setminus S$.
\begin{prop}
\label{prop4.5}
For $x\in\mathbb{T}^3$ such that $r(x)$ is bounded away from $\pm 1$, we have the following asymptotics for the (scaled) two point correlation function:
\begin{equation*}
K_2(x)=\frac{4}{\pi^2}+L_2(x)+\epsilon(x)
\end{equation*}
where
\begin{multline}
\label{L2}
L_2(x):=
\frac{4}{\pi^2}
\left[
\frac{1}{2}r^2
+\frac{tr(X)}{3}
+\frac{tr(Y^2)}{18}
+\frac{3}{8}r^4
-\frac{tr(XY^2)}{45}
-\frac{tr(X^2)}{45}
\right.
\\
\left.
+\frac{tr(Y^4)}{900}
+\frac{tr(Y^2)^2}{1800}
-\frac{tr(X)tr(Y^2)}{90}
+\frac{1}{6}r^2tr(X)
+\frac{1}{36}r^2tr(Y^2)
\right]
\end{multline}
and
\begin{equation*}
\epsilon(x):=O[r^6+tr(X^3)+tr(Y^6)].
\end{equation*}
\end{prop}
\begin{proof}[Proof of Proposition \ref{prop4.5} assuming Lemma \ref{lemma5.1}]
By Proposition \ref{explicit}, we have \eqref{expl}; for the first factor of \eqref{expl}, as $r(x)$ is bounded away from $\pm 1$, we may use the expansion \eqref{tay1}. On the second factor of \eqref{expl}, apply Lemma \ref{lemma5.1} with $X,Y$ as in \eqref{X} and \eqref{Y}. 
\end{proof}
\noindent
Later we will need to integrate $L_2$ term-wise.
\\
{\bf Notation.} To simplify the formulation of our next result, We will write $f\sim_\psi g$ if
\begin{equation*}
|f-g|=O\left(\frac{|\mathcal{X}(4)|}{\mathcal{N}^4}+\frac{|\mathcal{C}(6)|}{\mathcal{N}^6}\right)
\end{equation*}
and we will write $f\sim_\varphi g$ if
\begin{equation*}
|f-g|=O\left(\frac{1}{m^{1/28-o(1)}\cdot\mathcal{N}^2}+\frac{|\mathcal{X}(4)|}{\mathcal{N}^4}+\frac{|\mathcal{C}(6)|}{\mathcal{N}^6}\right).
\end{equation*}
\begin{lemma}
\label{lemma4.6}
We have the following estimates:
\begin{small}
\begin{enumerate}
\begin{multicols}{2}
\item
\begin{equation*}
\int_{\mathbb{T}^3}tr X(x)dx\sim_\psi-\frac{3}{\mathcal{N}}-\frac{3}{\mathcal{N}^2}.
\end{equation*}
\item
\begin{equation*}
\int_{\mathbb{T}^3}tr Y^2(x)dx\sim_\psi\frac{9}{\mathcal{N}}-\frac{6}{\mathcal{N}^2}.
\end{equation*}
\item
\begin{equation*}
\int_{\mathbb{T}^3}tr (XY^2)(x)dx\sim_\psi-\frac{9}{\mathcal{N}^2}.
\end{equation*}
\item
\begin{equation*}
\int_{\mathbb{T}^3}tr (X^2)(x)dx\sim_\psi\frac{15}{\mathcal{N}^2}.
\end{equation*}
\item
\begin{equation*}
\int_{\mathbb{T}^3}tr (Y^4)(x)dx\sim_\varphi\frac{351}{5}\cdot\frac{1}{\mathcal{N}^2}.
\end{equation*}
\item
\begin{equation*}
\int_{\mathbb{T}^3}(tr Y^2(x))^2dx\sim_\varphi\frac{567}{5}\cdot\frac{1}{\mathcal{N}^2}.
\end{equation*}
\item
\begin{equation*}
\int_{\mathbb{T}^3}(tr X\cdot tr Y^2)(x)dx\sim_\psi-\frac{27}{\mathcal{N}^2}.
\end{equation*}
\item
\begin{equation*}
\int_{\mathbb{T}^3}(r^2tr X)(x)dx\sim_\psi-\frac{3}{\mathcal{N}^2}.
\end{equation*}
\item
\begin{equation*}
\int_{\mathbb{T}^3}(r^2tr (Y^2))(x)dx\sim_\psi\frac{15}{\mathcal{N}^2}.
\end{equation*}
\item
\begin{equation*}
\int_{\mathbb{T}^3}tr (X^3)(x)dx=O\left(\frac{|\mathcal{C}(6)|}{\mathcal{N}^6}\right).
\end{equation*}
\item
\begin{equation*}
\int_{\mathbb{T}^3}tr (Y^6)(x)dx=O\left(\frac{|\mathcal{C}(6)|}{\mathcal{N}^6}\right).
\end{equation*}
\end{multicols}
\end{enumerate}
\end{small}
\end{lemma}
\noindent
The proof of Lemma \ref{lemma4.6} is given in Section \ref{technical}.

\subsection{Proof of Proposition \ref{theprop}}
\label{pfthm}
Assuming the above preparatory results, we arrive at the asymptotics for the nodal area variance.
\begin{proof}[Proof of Proposition \ref{theprop}]
In the expression for the variance of Proposition \ref{precisevar}, we separate the domain of integration over the singular set $S\subset\mathbb{T}^3$ of Definition \ref{S} and its complement:
\begin{equation}
\label{pfapproxvar0}
Var(\mathcal{A})
=
\frac{E}{3}\int_{\mathbb{T}^3\setminus S}\left(K_2(x)-\frac{4}{\pi^2}\right)dx
+
\frac{E}{3}\int_{S}\left(K_2(x)-\frac{4}{\pi^2}\right)dx.
\end{equation}
By Lemma \ref{rupper}, the asymptotics for $K_2$ of Proposition \ref{prop4.5} hold outside the singular set:
\begin{equation}
\label{pfapproxvar1}
\int_{\mathbb{T}^3\setminus S}\left(K_2(x)-\frac{4}{\pi^2}\right)dx
=
\int_{\mathbb{T}^3\setminus S}L_2(x)dx
+
O\int_{\mathbb{T}^3\setminus S}|\epsilon(x)|dx.
\end{equation}
Note that the constant term $4/\pi^2$ of the nodal area variance cancels out with the expectation squared. Next, recall Proposition \ref{singset}:
\begin{equation}
\label{pfapproxvar2}
\int_{S}|K_2(x)|dx
\ll
meas(S)
\ll
\mathcal{R}(6)
=
\frac{|\mathcal{C}(6)|}{\mathcal{N}^6}.
\end{equation}
Inserting \eqref{pfapproxvar1} and \eqref{pfapproxvar2} into \eqref{pfapproxvar0} gives
\begin{equation}
\label{pfapproxvar3}
Var(\mathcal{A})
=
\frac{E}{3}\int_{\mathbb{T}^3\setminus S}L_2(x)dx
+
E
\left(
O
\left(
\int_{\mathbb{T}^3\setminus S}|\epsilon(x)|dx
\right)
+
O
\left(\frac{|\mathcal{C}(6)|}{\mathcal{N}^6}\right)
\right).
\end{equation}
The former error term is redundant by Lemma \ref{lemma4.6}, parts 10 and 11.
Using $|r(x)|\leq 1$ and Lemma \ref{O(1)} in the expression \eqref{L2} for $L_2$, we get
\begin{equation*}
\label{useO(1)}
\left|\int_{\mathbb{T}^3\setminus S}L_2(x)dx
-
\int_{\mathbb{T}^3}L_2(x)dx\right|
\ll
\int_{S}|L_2(x)|dx
\ll
meas(S)
\end{equation*}
which together with \eqref{pfapproxvar3} and \eqref{pfapproxvar2} implies
\begin{equation}
\label{approxvar}
Var(\mathcal{A})=\frac{E}{3}\int_{\mathbb{T}^3}L_2(x)dx+O\left(E\cdot\frac{|\mathcal{C}(6)|}{\mathcal{N}^6}\right).
\end{equation}
We integrate \eqref{approxvar} term-wise (recall the expression \eqref{L2} for $L_2$), and, as the integral is over the whole torus, we may apply the considerations
\begin{equation*}
\int_{\mathbb{T}^3}r^2(x)dx=\frac{1}{\mathcal{N}},
\qquad
\int_{\mathbb{T}^3}r^4(x)dx=\frac{3}{\mathcal{N}^2}+O\left(\frac{1}{\mathcal{N}^3}+\frac{|\mathcal{X}(4)|}{\mathcal{N}^4}\right)
\end{equation*}
(see Lemma \ref{lemma5.4}), and the estimates of Lemma \ref{lemma4.6}, to deduce:
\begin{multline*}
Var(\mathcal{A})
=\frac{E}{3}\frac{4}{\pi^2}\int_{\mathbb{T}^3}
\left[
\frac{1}{2}r^2+\frac{tr(X)}{3}+\frac{tr(Y^2)}{18}
+\frac{3}{8}r^4-\frac{tr(XY^2)}{45}-\frac{tr(X^2)}{45}+\frac{tr(Y^4)}{900}
\right.
\\
\left.
+\frac{tr(Y^2)^2}{1800}
-\frac{tr(X)tr(Y^2)}{90}
+\frac{1}{6}r^2tr(X)+\frac{1}{36}r^2tr(Y^2)
\right]dx
+O\left(E\cdot\frac{|\mathcal{C}(6)|}{\mathcal{N}^6}\right)
\\=\frac{E}{3}\frac{4}{\pi^2}
\left[
\frac{1}{\mathcal{N}}
\left(
\frac{1}{2}
-
\frac{1}{3}
\cdot(-3)
+
\frac{1}{18}
\cdot 9
\right)
\right.
\\
\left.
+
\frac{1}{\mathcal{N}^2}
\left(
\frac{1}{3}
\cdot(-3)
+
\frac{1}{18}
\cdot(-6)
+
\frac{3}{8}\cdot 3
-
\frac{1}{45}(-9)
-
\frac{1}{45}\cdot 15
+
\frac{1}{900}\cdot\frac{351}{5}
\right.
\right.
\\
\left.
\left.
+
\frac{1}{1800}\cdot\frac{567}{5}
-
\frac{1}{90}(-27)
+\frac{1}{6}(-3)+\frac{1}{36}\cdot 15\right)
\right.
\\
\left.
+O\left(\frac{1}{m^{1/28-o(1)}\cdot\mathcal{N}^2}+\frac{|\mathcal{X}(4)|}{\mathcal{N}^4}+\frac{|\mathcal{C}(6)|}{\mathcal{N}^6}\right)
\right],
\end{multline*}
where we note the error term $m/\mathcal{N}^3$ is negligible. The terms of order $m/\mathcal{N}$ cancel perfectly: as noted in the introduction, the $3$-dimensional torus exhibits arithmetic Berry cancellation (see the next section for more details). The terms of order $m/\mathcal{N}^2$ sum up to
\begin{equation*}
\frac{E}{3}\cdot\frac{4}{\pi^2}\cdot\frac{1}{\mathcal{N}^2}\cdot\frac{6}{375}
=
\frac{32}{375}\cdot\frac{m}{\mathcal{N}^2},
\end{equation*}
hence, recalling \eqref{totnumlp}, the claim of the present proposition.
\end{proof}

\subsection{A note on arithmetic Berry cancellation}
\label{aBc}
Let us analyse in more detail the vanishing of the term of order $m/\mathcal{N}$ of the nodal area variance (cf. \cite[Section 4.2]{krkuwi}). The leading term of $K_2(x)-4/\pi^2$ is (recall \eqref{L2}, \eqref{X} and \eqref{Y})
\begin{equation*}
\frac{4}{\pi^2}\left[\frac{1}{2}r^2+\frac{tr(X)}{3}
+\frac{tr(Y^2)}{18}\right]
\sim
\frac{4}{\pi^2}\left[\frac{1}{2}r^2+\frac{1}{3}\left(\frac{3}{E}DD^t\right)
+\frac{1}{18}\left(\frac{9}{E^2}tr(H^2)\right)
\right]=\frac{2}{\pi^2}v(x),
\end{equation*}
having defined
\begin{equation*}
v(x):=r^2(x)-\frac{2}{E}(DD^t)(x)+\frac{1}{E^2}tr(H^2(x)).
\end{equation*}
The latter expression has the same shape as the two-dimensional case \cite[(39)]{krkuwi}: the remainder of this discussion is essentially identical to \cite[Section 4.2]{krkuwi}. One rewrites
\begin{equation*}
v(x)=\frac{4}{N^2}\sum_{\mu_1,\mu_2\in\mathcal{E}}e([\mu_1+\mu_2]\cdot x)\cdot\cos^4\left(\frac{\varphi_{\mu_1,\mu_2}}{2}\right),
\end{equation*}
where $\varphi_{\mu_1,\mu_2}$ is the angle between the two lattice points $\mu_1,\mu_2$. On integrating over the torus \eqref{precisevarexpression}, all summands such that $\mu_1+\mu_2\neq 0$ vanish (see also \eqref{circlemethod} to follow). As $\varphi_{\mu_1,-\mu_1}=\pi$, the arithmetic cancellation phenomenon is tantamount to $\cos^4(\varphi/2)$ vanishing at $\pi$, similarly to the two-dimensional problem.





\section{The leading term of the variance: proof of Lemma \ref{lemma4.6}}
\label{technical}
\subsection{Preparatory results}
\label{S4stru}
Recall the expression of the covariance function \eqref{r} and its derivatives \eqref{D} and \eqref{H}; also recall the notation of Definition \ref{correl} for the set of lattice point correlations.
\begin{lemma}
\label{lemma5.4}
 We have the following estimates, where $\sim_\rho$ means up to an error
\begin{equation*}
O\left(\frac{1}{\mathcal{N}^3}+\frac{|\mathcal{X}(4)|}{\mathcal{N}^4}\right)
\end{equation*}
and $\sim_\sigma$ means up to an error 
\begin{equation*}
O\left(\frac{1}{m^{1/28-o(1)}\cdot\mathcal{N}^2}+\frac{|\mathcal{X}(4)|}{\mathcal{N}^4}\right):
\end{equation*}
\begin{small}
\begin{enumerate}
\begin{multicols}{2}
\item
\begin{gather*}
\int_{\mathbb{T}^3}r^2(x)dx=\frac{1}{\mathcal{N}};
\\
\int_{\mathbb{T}^3}r^4(x)dx\sim_\rho\frac{3}{\mathcal{N}^2}.
\end{gather*}
\item
\begin{gather*}
\frac{1}{E}\int_{\mathbb{T}^3}(DD^t)(x)dx=\frac{1}{\mathcal{N}};
\\
\frac{1}{E^2}\int_{\mathbb{T}^3}(DD^t)^2(x)dx\sim_\rho\frac{5}{3}\cdot\frac{1}{\mathcal{N}^2}.
\end{gather*}
\item
\begin{equation*}
\frac{1}{E}\int_{\mathbb{T}^3}(r^2DD^t)(x)dx\sim_\rho\frac{1}{\mathcal{N}^2}.
\end{equation*}
\item
\begin{gather*}
\frac{1}{E^2}\int_{\mathbb{T}^3}tr(H^2(x))dx=\frac{1}{\mathcal{N}};
\\
\frac{1}{E^2}\int_{\mathbb{T}^3}(r^2tr(H^2))(x)dx\sim_\rho\frac{5}{3}\cdot\frac{1}{\mathcal{N}^2}.
\end{gather*}
\item
\begin{gather*}
\frac{1}{E^4}\int_{\mathbb{T}^3}tr(H^4(x))dx\sim_\sigma\frac{13}{15}\cdot\frac{1}{\mathcal{N}^2};
\\
\frac{1}{E^4}\int_{\mathbb{T}^3}tr(H^2(x))^2dx\sim_\sigma\frac{7}{5}\cdot\frac{1}{\mathcal{N}^2}.
\end{gather*}
\item
\begin{equation*}
\frac{1}{E^3}\int_{\mathbb{T}^3}(DD^ttr(H^2))(x)dx\sim_\rho\frac{1}{\mathcal{N}^2}.
\end{equation*}
\item
\begin{equation*}
\frac{1}{E^2}\int_{\mathbb{T}^3}(rDHD^t)(x)dx\sim_\rho-\frac{1}{3}\cdot\frac{1}{\mathcal{N}^2}.
\end{equation*}
\item
\begin{equation*}
\frac{1}{E^3}\int_{\mathbb{T}^3}(DH^2D^t)(x)dx\sim_\rho\frac{1}{3}\cdot\frac{1}{\mathcal{N}^2}.
\end{equation*}
\item
\begin{equation*}
\frac{1}{E^3}\int_{\mathbb{T}^3}(DD^t)^3(x)dx\ll\frac{|\mathcal{C}(6)|}{\mathcal{N}^6}.
\end{equation*}
\item
\begin{equation*}
\frac{1}{E}\int_{\mathbb{T}^3}(r^4DD^t)(x)dx\ll\frac{|\mathcal{C}(6)|}{\mathcal{N}^6}.
\end{equation*}
\item
\begin{equation*}
\frac{1}{E^6}\int_{\mathbb{T}^3}tr(H^6(x))dx\ll\frac{|\mathcal{C}(6)|}{\mathcal{N}^6}.
\end{equation*}
\end{multicols}
\item
\begin{equation*}
\frac{1}{E^3}\int_{\mathbb{T}^3}(rDD^tDHD^t)(x)dx\ll\frac{|\mathcal{C}(6)|}{\mathcal{N}^6}.
\end{equation*}
\end{enumerate}
\end{small}
\end{lemma}
\begin{proof}
The various estimates are obtained with the following common strategy. Firstly, one rewrites the integrand using the expressions \eqref{r}, \eqref{D} and \eqref{H} for the covariance function and its (first and second order) derivatives. Next, the integral over the torus in taken, invoking the orthogonality relations of the exponentials:
\begin{equation}
\label{circlemethod}
\int_{\mathbb{T}^3}e(\mu\cdot x)dx=
\begin{cases}
1 \quad \mu=0
\\
0 \quad \mu\neq 0.
\end{cases}
\end{equation}
We are thus left with a summation over the set of $\ell$-correlations $\mathcal{C}(\ell)$, where $\ell=2,4
$ or $6$. The summands are certain products of inner products between two lattice points. The summations involving $2$-correlations are computed directly, and for $k=
6$ we need only an upper bound. The most delicate computations are for $4$-correlations, when we split the summation exploiting the structure of $\mathcal{C}(4)$ (see \eqref{splitsums}). This leads to computing $k$-th moments (for $k=1, 2, 3,$ or $4$) of the normalised inner product of two lattice points, applying Lemma \ref{innerk}.

We now present the details of the proof for some of the estimates of the present lemma; the remaining computations apply the same ideas (outlined above), and we will omit them here. We begin with part 1, first statement, which is an immediate consequence of \eqref{C2}:
\begin{equation*}
\int_{\mathbb{T}^3}r^2(x)dx=\mathcal{R}(2)=\frac{|\mathcal{C}(2)|}{\mathcal{N}^2}=\frac{1}{\mathcal{N}}.
\end{equation*}
The second statement of part 1 follows from the structure of $\mathcal{C}(4)$ \eqref{splitsums}:
\begin{equation*}
\int_{\mathbb{T}^3}r^4(x)dx=\mathcal{R}(4)=\frac{|\mathcal{C}(4)|}{\mathcal{N}^4}=\frac{3}{\mathcal{N}^2}+O\left(\frac{1}{\mathcal{N}^3}+\frac{|\mathcal{X}(4)|}{\mathcal{N}^4}\right).
\end{equation*}
Let us show part 2 of the present lemma, starting with the first statement. By \eqref{D}, we may rewrite the integrand as
\begin{equation}
\label{eqDDt}
DD^t=tr(D^tD)
=-\frac{4\pi^2}{\mathcal{N}^2}\cdot
\sum_{\mu_1,\mu_2}e([\mu_1+\mu_2]\cdot x)(\mu_1\cdot\mu_2).
\end{equation}
We take the integral over $\mathbb{T}^3$, bearing in mind \eqref{circlemethod}, and compute the resulting summation over the set of $2$-correlations, using \eqref{C2}:
\begin{equation*}
\int_{\mathbb{T}^3}(DD^t)(x)dx
=
-\frac{4\pi^2}{\mathcal{N}^2}\cdot
\sum_{\mathcal{C}(2)}(\mu_1\cdot\mu_2)
=
-\frac{4\pi^2}{\mathcal{N}^2}\cdot
\sum_{\mu_2}(-\mu_2\cdot\mu_2)
=
\frac{E}{\mathcal{N}},
\end{equation*}
as claimed. For the second statement of part 2, we begin by squaring \eqref{eqDDt}:
\begin{gather*}
(DD^t)^2=\frac{(4\pi^2)^2}{\mathcal{N}^4}\cdot
\sum_{\mathcal{E}_m^4}
e([\mu_1+\mu_2+\mu_3+\mu_4]\cdot x)
\cdot
(\mu_1\cdot\mu_2)
\cdot
(\mu_3\cdot\mu_4).
\end{gather*}
By \eqref{circlemethod},
\begin{equation}
\label{thesum}
\int_{\mathbb{T}^3}(DD^t)^2dx
=
\frac{(4\pi^2)^2}{\mathcal{N}^4}
\cdot\sum_{
\mathcal{C}(4)}
(\mu_1\cdot\mu_2)
\cdot
(\mu_3\cdot\mu_4).
\end{equation}
To treat the resulting summation over $4$-correlations, we split it with \eqref{splitsums}. The contribution over diagonal and non-degenerate quadruples is bounded via Cauchy-Schwartz:
\begin{equation*}
\sum_{
\mathcal{D''}\dot{\cup}\mathcal{X}(4)}
(\mu_1\cdot\mu_2)
\cdot
(\mu_3\cdot\mu_4)
\leq
\sum_{
\mathcal{D''}\dot{\cup}\mathcal{X}(4)}
(\sqrt{m})^4
\ll
m^2\cdot
(\mathcal{N}+|\mathcal{X}(4)|).
\end{equation*}
There are three more contributions to the summation in \eqref{thesum}, that arise from symmetric (and non-diagonal) $4$-correlations; we directly compute the first of these contributions:
\begin{equation*}
\sum_{\substack{\mu_1=-\mu_2\\\mu_3=-\mu_4}}
(\mu_1\cdot\mu_2)
\cdot
(\mu_3\cdot\mu_4)
=
\sum_{\mu_2,\mu_4}
(-\mu_2\cdot\mu_2)
\cdot
(-\mu_4\cdot\mu_4)
=
m^2\mathcal{N}^2.
\end{equation*}
For the remaining two summations, we invoke Lemma \ref{innerk} with $k=2$:
\begin{equation*}
\sum_{\substack{\mu_1=-\mu_3\\\mu_2=-\mu_4}}
(\mu_1\cdot\mu_2)
\cdot
(\mu_3\cdot\mu_4)
=
\sum_{\substack{\mu_1=-\mu_4\\\mu_2=-\mu_3}}
(\mu_1\cdot\mu_2)
\cdot
(\mu_3\cdot\mu_4)
=
\sum_{\mu_4}\sum_{\mu_3}
(\mu_3\cdot\mu_4)^2
=
\frac{m^2\mathcal{N}^2}{3}.
\end{equation*}
The various contributions yield
\begin{equation}
\label{contrib}
\sum_{\mathcal{C}(4)}
(\mu_1\cdot\mu_2)
\cdot
(\mu_3\cdot\mu_4)
\\=
\frac{5}{3}\cdot m^2\mathcal{N}^2
+
O
(m^2 \mathcal{N})
+
O
(m^2\cdot|\mathcal{X}(4)|).
\end{equation}
Inserting \eqref{contrib} into \eqref{thesum} we arrive at the second statement of part 2 of the present lemma. The proof of part 3 is very similar to that of part 2, second statement, except Lemma \ref{innerk} is applied with $k=1$.

To prove part 4, first statement, recall \eqref{H} and \eqref{circlemethod} to directly compute
\begin{equation*}
\int_{\mathbb{T}^3}tr(H^2(x))dx
=
\frac{(4\pi^2)^2}{\mathcal{N}^2}\cdot
\sum_{\mathcal{C}(2)}
tr(\mu_1^t\mu_1\mu_2^t\mu_2)
=
\frac{(4\pi^2)^2}{\mathcal{N}^2}\cdot
\sum_{\mu_1}
(\mu_1\cdot\mu_1)^2
=
\frac{E^2}{\mathcal{N}}.
\end{equation*}
For part 4, second statement, \eqref{r}, \eqref{H} and \eqref{circlemethod} imply
\begin{equation*}
\int_{\mathbb{T}^3}(r^2tr(H^2))(x)dx
=
\frac{(4\pi^2)^2}{\mathcal{N}^4}\cdot
\sum_{\mathcal{C}(4)}
tr({\mu_3}^t{\mu_3}{\mu_4}^t{\mu_4})
=
\frac{(4\pi^2)^2}{\mathcal{N}^4}\cdot
\sum_{\mathcal{C}(4)}
(\mu_3\cdot\mu_4)^2;
\end{equation*}
one now splits the sum and proceeds as in the proof of part 2.

Let us prove part 5 of the present lemma, first statement. By \eqref{H} and \eqref{circlemethod}, we have
\begin{multline*}
\int_{\mathbb{T}^3}tr(H^4(x))dx=\frac{(4\pi^2)^4}{\mathcal{N}^4}
\sum_{\mathcal{C}(4)} tr({\mu_1}^t\mu_1{\mu_2}^t\mu_2{\mu_3}^t\mu_3{\mu_4}^t\mu_4)
\allowdisplaybreaks[4]
\\=
\frac{(4\pi^2)^4}{\mathcal{N}^4}
\left[
\sum_{\mu_2,\mu_4}tr({\mu_2}^t\mu_2{\mu_2}^t\mu_2{\mu_4}^t\mu_4{\mu_4}^t\mu_4)
+
\sum_{\mu_3,\mu_4}tr({\mu_3}^t\mu_3{\mu_4}^t\mu_4{\mu_3}^t\mu_3{\mu_4}^t\mu_4)
\right.
\\
\left.
+
\sum_{\mu_3,\mu_4}tr({\mu_4}^t\mu_4{\mu_3}^t\mu_3{\mu_3}^t\mu_3{\mu_4}^t\mu_4)
\right]
+E^4\cdot O\left(\frac{1}{\mathcal{N}^3}+\frac{|\mathcal{X}(4)|}{\mathcal{N}^4}\right)
\\=
\frac{(4\pi^2)^4}{\mathcal{N}^4}
\left[
\sum_{\mu_2,\mu_4}m^2(\mu_2\cdot\mu_4)^2
+
\sum_{\mu_3,\mu_4}(\mu_3\cdot\mu_4)^4
+
\sum_{\mu_3,\mu_4}m^2(\mu_3\cdot\mu_4)^2
\right]
\\+E^4\cdot O\left(\frac{1}{\mathcal{N}^3}+\frac{|\mathcal{X}(4)|}{\mathcal{N}^4}\right).
\end{multline*}
One computes the three summations on the RHS of the latter expression via Lemma \ref{innerk}, with $k=2, 4$:
\begin{gather*}
\int_{\mathbb{T}^3}tr(H^4(x))dx=
\frac{E^4}{\mathcal{N}^2}
\left[
\frac{1}{3}
+
\frac{1}{5}
+
\frac{1}{3}
+
O\left(\frac{1}{m^{1/28-o(1)}}\right)
+O\left(\frac{|\mathcal{X}(4)|}{\mathcal{N}^2}\right)
\right],
\end{gather*}
where we note the error term $E^4/\mathcal{N}^3$ is negligible by \eqref{totnumlp}. The second statement of part 5, and parts 6, 7 and 8 of the present lemma are all derived in a similar fashion, and we will omit these proofs here.

Let us prove part 12 of the present lemma, parts 9, 10 and 11 being similar. By \eqref{r}, \eqref{D}, \eqref{H} and \eqref{circlemethod},
\begin{equation*}
\int_{\mathbb{T}^3}(rDD^tDHD^t)(x)dx
=
-\frac{(4\pi^2)^3}{\mathcal{N}^6}
\sum_{\mathcal{C}(6)}
(\mu_1\cdot\mu_2)
\cdot
(\mu_3\cdot\mu_4)
\cdot
(\mu_4\cdot\mu_5)
\ll
\frac{E^3}{\mathcal{N}^6}\cdot |\mathcal{C}(6)|
\end{equation*}
(for summations over 
$6$-correlations, an upper bound via the Cauchy-Schwartz inequality is sufficient for our purposes).
\end{proof}

\subsection{Proof of Lemma \ref{lemma4.6}}
\begin{proof}
[Proof of Lemma \ref{lemma4.6}] 
To prove part 1, recall Lemma \ref{O(1)} (uniform boundedness of $X$) and write
\begin{equation*}
\int_{\mathbb{T}^3}tr X(x)dx
=
\int_{\mathbb{T}^3\setminus S}tr X(x)dx
+
O(meas \ S).
\end{equation*}
Recall the expression of $X$ \eqref{X}; one uses the approximation \eqref{tay2} on $\mathbb{T}^3\setminus S$, and Proposition \ref{singset} to bound the contribution of the singular set:
\begin{multline*}
\int_{\mathbb{T}^3}tr X(x)dx
=
-\frac{3}{E}
\left(
\int_{\mathbb{T}^3}DD^tdx
+
\int_{\mathbb{T}^3}r^2DD^tdx
\right)
\\+
O
\left(
\frac{1}{E}
\int_{\mathbb{T}^3}r^4DD^tdx
\right)
+
O\left(\frac{|\mathcal{C}(6)|}{\mathcal{N}^6}\right).
\end{multline*}
To compute the three integrals on the RHS of the latter expression, apply Lemma \ref{lemma5.4}, parts 2, 3 and 10. Here and elsewhere the error term $1/\mathcal{N}^3$ (arising from several of the estimates of Lemma \ref{lemma5.4}) is negligible compared to $|\mathcal{C}(6)|/\mathcal{N}^6$. Part 2 of the present lemma is derived in a similar way.

Let us show part 3 of the present lemma, parts 4, 7, 8 and 9 being similar. By Lemma \ref{O(1)}, \eqref{tay2} and Proposition \ref{singset},
\begin{equation*}
\int_{\mathbb{T}^3}tr(XY^2)(x)dx
=
-\frac{27}{E^3}
\left[
\int_{\mathbb{T}^3}tr(DH^2D^t)dx
+
O
\int_{\mathbb{T}^3}rDD^tDHD^tdx
\right]
+
O\left(\frac{|\mathcal{C}(6)|}{\mathcal{N}^6}\right).
\end{equation*}
Now Lemma \ref{lemma5.4}, parts 8 and 12, yields
\begin{equation*}
\int_{\mathbb{T}^3}tr(XY^2)(x)dx
=
\frac{1}{\mathcal{N}^2}
\left[
-9
+O\left(\frac{|\mathcal{X}(4)|}{\mathcal{N}^2} 
+\frac{|\mathcal{C}(6)|}{\mathcal{N}^4}\right)
\right],
\end{equation*}
which concludes the proof of part 3 of the present lemma.

We now prove part 5, part 6 being similar. By Lemma \ref{O(1)} and Proposition \ref{singset}, we have
\begin{equation*}
\int_{\mathbb{T}^3}tr(Y^4)(x)dx
=
\frac{81}{E^4}
\int_{\mathbb{T}^3}tr(H^4)dx
+
O\left(\frac{|\mathcal{C}(6)|}{\mathcal{N}^6}\right).
\end{equation*}
Now Lemma \ref{lemma5.4}, part 5 yields
\begin{equation*}
\int_{\mathbb{T}^3}tr(Y^4)(x)dx
=
\frac{351}{5}\cdot\frac{1}{\mathcal{N}^2}
+O\left(\frac{1}{m^{1/28-o(1)}\cdot\mathcal{N}^2}+\frac{|\mathcal{X}(4)|}{\mathcal{N}^4}+\frac{|\mathcal{C}(6)|}{\mathcal{N}^6}\right),
\end{equation*}
hence the claim of part 5 of the present lemma.
Lastly, we show part 10, part 11 being similar. By Lemma \ref{O(1)} and Proposition \ref{singset}, we have
\begin{equation*}
\int_{\mathbb{T}^3}tr(X^3)(x)dx
=
-\frac{27}{E^3}
\int_{\mathbb{T}^3}(DD^t)^3(x)dx
+
O\left(\frac{|\mathcal{C}(6)|}{\mathcal{N}^6}\right)
\ll\frac{|\mathcal{C}(6)|}{\mathcal{N}^6},
\end{equation*}
where in the last step we applied Lemma \ref{lemma5.4}, part 9.
\end{proof}



\appendix
\section{Berry's method: proof of Lemma \ref{lemma5.1}}
\label{Berry's}
\noindent
In this section, we establish Lemma \ref{lemma5.1}: following \cite{berry2} and \cite{krkuwi}, we regard $\mathbb{E}[\|w_1\|\|w_2\|]$ (recall the notation in the statement of the lemma) as a function of the entries of the matrices $X$ \eqref{X} and $Y$ \eqref{Y}, and perform a Taylor expansion about $X=Y=0$. We employ Berry's method as opposed to computing the Taylor polynomial by brute force, which would result in a longer computation.
\begin{lemma}
\label{I+J lemma}
Let:
$w_1,w_2\in\mathbb{R}^3$, $(w_1,w_2)\sim N(0,\Omega)$ with
$
\Omega=I_6+
\begin{pmatrix}
X & Y \\ Y & X
\end{pmatrix}.
$
Then
\begin{equation}
\label{im1}
\mathbb{E}[\|w_1\|\|w_2\|]=
\frac{1}{2\pi}\iint_{\R^{2}_{+}}
(f(0,0)-f(t,0)-f(0,s)+f(t,s))
\frac{dtds}{(ts)^{\frac{3}{2}}}
\end{equation}
with
\begin{equation}
\label{im2}
f(t,s)=
\frac{1}{\sqrt{\det{(I_6+J(t,s))}}},
\end{equation}
where
\begin{equation}
\label{im3}
I_6+J=
\begin{pmatrix}
(1+t)I_3+tX & \sqrt{ts}Y \\ \sqrt{ts}Y & (1+s)I_3+sX
\end{pmatrix}
\end{equation}
is a perturbation of the identity matrix $I_6$.
\end{lemma}
\begin{proof}
We begin with \cite[(24)]{berry2}:
\begin{equation*}
\|w_i\|=
\frac{1}{\sqrt{2\pi}} \int_{\R_{+}}
\left(1-e^{-t\|w_i\|^2/2}\right)\frac{dt}{t^{\frac{3}{2}}},
\qquad
i=1,2.
\end{equation*}
The LHS of \eqref{im1} becomes
\begin{multline*}
\mathbb{E}[\|w_1\|\|w_2\|]=
\frac{1}{2\pi}\iint_{\R^{2}_{+}} \mathbb{E}\left[(1-e^{-t\|w_i\|^2/2})(1-e^{-s\|w_i\|^2/2})\right]
\frac{dtds}{(ts)^{\frac{3}{2}}}
\\=
\frac{1}{2\pi}\iint_{\R^{2}_{+}}
\left(
\mathbb{E}[1]
+
\mathbb{E}\left[-e^{-t\|w_1\|^2/2}\right]
+
\mathbb{E}\left[-e^{-s\|w_2\|^2/2}\right]
+
\mathbb{E}\left[e^{-(t\|w_1\|^2+s\|w_2\|^2)/2}\right]
\right)
\frac{dtds}{(ts)^{\frac{3}{2}}}.
\end{multline*}
Setting
\begin{equation*}
f(t,s)=f_{X,Y}(t,s):=\mathbb{E}[\exp(-(t\|w_1\|^2+s\|w_2\|^2)/2)],
\end{equation*}
it remains to show that $f(t,s)$ may be rewritten as in \eqref{im2}. By definition of expectation,
\begin{multline*}
f(t,s)=
\int_{\R^{3}\times\R^{3}}
\frac{1}{\sqrt{(2\pi)^6}}\cdot\frac{1}{\sqrt{\det{\Omega}}}
\cdot
\\
\cdot
\exp(-(t\|w_1\|^2+s\|w_2\|^2)/2)
\cdot
\exp\left(-\frac{1}{2}
\begin{pmatrix}
w_1 & w_2
\end{pmatrix}
\Omega^{-1}
\begin{pmatrix}
w_1 \\ w_2
\end{pmatrix}
\right)
dw_1dw_2
\\=
\frac{1}{\sqrt{(2\pi)^6\det{\Omega}}}
\int_{\R^{3}\times\R^{3}}
\exp\left(-\frac{1}{2}
\begin{pmatrix}
w_1 & w_2
\end{pmatrix}
\left[\begin{pmatrix}
tI_3 & 0 \\ 0 & sI_3
\end{pmatrix}
+
\Omega^{-1}
\right]
\begin{pmatrix}
w_1 \\ w_2
\end{pmatrix}
\right)
dw_1dw_2
\\=
\frac{1}{\sqrt{\det{\Omega}}}
\cdot\sqrt{\det{
\left[\begin{pmatrix}
tI_3 & 0 \\ 0 & sI_3
\end{pmatrix}
+
\Omega^{-1}\right]^{-1}
}}
=
\frac{1}{\sqrt{\det{(I_6+J(t,s))}}},
\end{multline*}
with $I_6+J(t,s)$ as in \eqref{im3}.
\end{proof}
We will need the following expansions for a square matrix $P$, as $P\to 0$ entry-wise:
\begin{equation}
\label{detto-1}
(I+P)^{-1}=I-P+O(P^2)
\end{equation}
and
\begin{equation}
\label{detto-1/2}
(\det(I+P))^{-\frac{1}{2}}=1-\frac{1}{2}trP+\frac{1}{4}tr(P^2)+\frac{1}{8}(trP)^2+O\left(\max_{i,j}(P^3)_{ij}\right).
\end{equation}

\begin{proof}[Proof of Lemma \ref{lemma5.1}]
By Lemma \ref{I+J lemma}, we get the expression \eqref{im1}, and require the Taylor expansion of $f_{X,Y}(t,s)=\det(I_6+J)^{-\frac{1}{2}}$ around $X=Y=0$. By \eqref{im3} and the formula for the determinant of a block matrix:
\begin{equation*}
\det(I_6+J)
=
\det((1+t)I_3+tX)
\cdot
\det[(1+s)I_3+sX-\sqrt{ts}Y((1+t)I_3+tX)^{-1}\sqrt{ts}Y],
\end{equation*}
hence
\begin{multline}
\label{2factors}
f_{X,Y}(t,s)=\det(I_6+J)^{-\frac{1}{2}}
\\=\det((1+t)I_3+tX)^{-\frac{1}{2}}
\cdot
\det[(1+s)I_3+sX-tsY((1+t)I_3+tX)^{-1}Y]^{-\frac{1}{2}}.
\end{multline}
Bearing in mind that $I_3$ and $X$ are $3\times 3$ matrices, we have
\begin{equation*}
\det((1+t)I_3+tX)
=(1+t)^3\cdot\det\left(I_3+\frac{t}{1+t}X\right)
\end{equation*}
and thus rewrite the first factor on the RHS of \eqref{2factors} as
\begin{equation}
\label{fac1}
\frac{1}{(1+t)^{\frac{3}{2}}}\det\left(I_3+\frac{t}{1+t}X\right)^{-\frac{1}{2}}.
\end{equation}
Since
\begin{multline*}
\det[(1+s)I_3+sX-tsY((1+t)I_3+tX)^{-1}Y]
\}
\\=
\det\left\{
(1+s)
\left[I_3+\frac{s}{1+s}X-\frac{ts}{(1+t)(1+s)}Y\left(I_3+\frac{t}{1+t}X\right)^{-1}Y\right]
\right\},
\end{multline*}
the second factor on the RHS of \eqref{2factors} equals
\begin{equation*}
\frac{1}{(1+s)^{\frac{3}{2}}}
\det\left[I_3+\frac{s}{1+s}X-\frac{ts}{(1+t)(1+s)}Y\left(I_3+\frac{t}{1+t}X\right)^{-1}Y\right]^{-\frac{1}{2}}
;
\end{equation*}
applying \eqref{detto-1} with $P=\frac{t}{1+t}X$, we further rewrite the second factor on the RHS of \eqref{2factors} as:
\begin{multline}
\label{fac2}
\frac{1}{(1+s)^{\frac{3}{2}}}\cdot\det\left[I+\frac{s}{1+s}X-\frac{ts}{(1+t)(1+s)}Y^2
\right.\\\left.
+\frac{t^2s}{(1+t)^2(1+s)}YXY+O(YX^2Y)\right]^{-\frac{1}{2}}.
\end{multline}
Next, we apply \eqref{detto-1/2} to both \eqref{fac1} and \eqref{fac2}, with $P=\frac{t}{1+t}X$ and
\begin{equation*}
P=\frac{s}{1+s}X-\frac{ts}{(1+t)(1+s)}Y^2+\frac{t^2s}{(1+t)^2(1+s)}YXY+O(YX^2Y)
\end{equation*}
respectively. The above computations on the two factors of \eqref{2factors} yield
\begin{multline}
\label{fexpand}
\allowdisplaybreaks[4]
f_{X,Y}(t,s)=\frac{1}{(1+t)^{\frac{3}{2}}(1+s)^{\frac{3}{2}}}
\cdot\left[
1
-
\frac{1}{2}
\left(
\frac{t}{1+t}
+
\frac{s}{1+s}
\right)
tr(X)
\right.
\\
\left.
+
\frac{1}{2}
\cdot\frac{ts}{(1+t)(1+s)}
tr(Y^2)
-
\frac{1}{2}
\cdot\frac{ts}{(1+t)(1+s)}
\left(
\frac{t}{1+t}
+
\frac{s}{1+s}
\right)
tr(XY^2)
\right.
\\
\left.
+
\left(
\frac{3}{8}
\frac{t^2}{(1+t)^2}
+
\frac{3}{8}
\frac{s^2}{(1+s)^2}
+
\frac{1}{4}
\frac{ts}{(1+t)(1+s)}
\right)
tr(X^2)
\right.
\\
\left.
+
\frac{1}{4}\cdot
\frac{t^2s^2}{(1+t)^2(1+s)^2}
tr(Y^4)+
\frac{1}{8}
\frac{t^2s^2}{(1+t)^2(1+s)^2}
tr(Y^2)^2
\right.
\\
\left.
-
\frac{1}{4}\cdot\frac{ts}{(1+t)(1+s)}
\left(
\frac{t}{1+t}
+
\frac{s}{1+s}
\right)
tr(X)tr(Y^2)
\right]
+
O(tr(X^3)+tr(Y^6)),
\end{multline}
where we have used the assumption $rank(X)=1$ so that $tr(X)^2=tr(X^2)$. The integrand in \eqref{im1} is
\begin{equation*}
h_{X,Y}(t,s):=f(0,0)-f(t,0)-f(0,s)+f(t,s);
\end{equation*}
to compute the Taylor polynomial for $h$ around $X=Y=0$, first note that, except for the terms in $1,tr(X),tr(X^2)$, the various terms in the expansion of $h$ are the same as those in the expansion of $f$: this is because each term in \eqref{fexpand}, save for those in $1,tr(X),tr(X^2)$, vanishes when $t=0$ or $s=0$. Next, we directly compute the terms in $1,tr(X),tr(X^2)$ of the expansion of $h$ to be respectively
\begin{equation*}
\left(1-\frac{1}{(1+t)^{3/2}}\right)
\cdot
\left(1-\frac{1}{(1+s)^{3/2}}\right),
\end{equation*}

\begin{equation*}
\frac{1}{2}
\left[
\frac{t}{(1+t)^{5/2}}
\left(1-\frac{1}{(1+s)^{3/2}}\right)
+
\frac{s}{(1+s)^{5/2}}
\left(1-\frac{1}{(1+t)^{3/2}}\right)
\right],
\end{equation*}
and
\begin{equation*}
-
\frac{3}{8}
\frac{t^2}{(1+t)^{7/2}}
\left(1-\frac{1}{(1+s)^{3/2}}\right)
-
\frac{3}{8}
\frac{s^2}{(1+s)^{7/2}}
\left(1-\frac{1}{(1+t)^{3/2}}\right)
+
\frac{1}{4}
\frac{ts}{(1+t)^{5/2}(1+s)^{5/2}}.
\end{equation*}
To perform the integration
\begin{equation}
\mathbb{E}[\|w_1\|\|w_2\|]=
\frac{1}{2\pi}\iint_{\R^{2}_{+}}
h(t,s)
\frac{dtds}{(ts)^{\frac{3}{2}}}
\end{equation}
term-wise, we need to improve the error term $O(tr(X^3)+tr(Y^6))$ in the expansion of $h$ so that it depends on $t$ and $s$, as
\begin{equation*}
\iint_{\R^{2}_{+}}
\frac{dtds}{(ts)^{\frac{3}{2}}}
\end{equation*}
is divergent at the origin. To do this, we note that, for all $X$ and $Y$, $h$ vanishes when $t=0$ or $s=0$; hence, for $t,s\geq 0$, we may write
\begin{equation*}
h_{X,Y}(t,s)=O_{X,Y}(ts).
\end{equation*}
We may then improve the error term in the expansion of $h$ to
\begin{equation*}
O\left(\min(t,1)\cdot\min(s,1)\cdot(tr(X^3)+tr(Y^6))\right).
\end{equation*}
Therefore,
\begin{align}
\label{hexpand}
\allowdisplaybreaks[4]
h^{X,Y}(t,s)&=
\left(1-\frac{1}{(1+t)^{3/2}}\right)
\cdot
\left(1-\frac{1}{(1+s)^{3/2}}\right)
\\
\notag
&+
\frac{1}{2}
\left[
\frac{t}{(1+t)^{5/2}}
\left(1-\frac{1}{(1+s)^{3/2}}\right)
+
\frac{s}{(1+s)^{5/2}}
\left(1-\frac{1}{(1+t)^{3/2}}\right)
\right]
tr(X)
\\
\notag
&+
\frac{1}{2}
\frac{t}{(1+t)^{5/2}}
\frac{s}{(1+s)^{5/2}}
tr(Y^2)
-
\frac{1}{2}
\left(
\frac{t^2}{(1+t)^{7/2}}
\frac{s}{(1+s)^{5/2}}
\right.
\\
\notag
&
\left.
+
\frac{t}{(1+t)^{5/2}}
\frac{s^2}{(1+s)^{7/2}}
\right)
tr(XY^2)
+
\left[
-
\frac{3}{8}
\frac{t^2}{(1+t)^{7/2}}
\left(1-\frac{1}{(1+s)^{3/2}}\right)
\right.
\\
\notag
&
\left.
-
\frac{3}{8}
\frac{s^2}{(1+s)^{7/2}}
\left(1-\frac{1}{(1+t)^{3/2}}\right)
+
\frac{1}{4}
\frac{ts}{(1+t)^{5/2}(1+s)^{5/2}}
\right]
tr(X^2)
\\
\notag
&+
\frac{1}{4}
\frac{t^2s^2}{(1+t)^{7/2}(1+s)^{7/2}}
tr(Y^4)
+
\frac{1}{8}
\frac{t^2s^2}{(1+t)^{7/2}(1+s)^{7/2}}
tr(Y^2)^2
\\
\notag
&
-
\frac{1}{4}
\left(
\frac{ts^2}{(1+t)^{5/2}(1+s)^{7/2}}
+
\frac{t^2s}{(1+t)^{7/2}(1+s)^{5/2}}
\right)
tr(X)tr(Y^2)
\\
\notag
&
+
O\left(\min(t,1)\cdot\min(s,1)\cdot(tr(X^3)+tr(Y^6))\right).
\end{align}
Lastly, we insert \eqref{hexpand} into \eqref{im1}, and compute the integrals
\begin{align*}
\int_{0}^{\infty}
\left(1-\frac{1}{(1+t)^{3/2}}\right)
\frac{dt}{t^{3/2}}=4,
&&
\int_{0}^{\infty}
\frac{dt}{(1+t)^{5/2}\sqrt{t}}=\frac{4}{3},
\\
\int_{0}^{\infty}
\frac{\sqrt{t} \ dt}{(1+t)^{7/2}}=\frac{4}{15},
&&
\int_{0}^{\infty}
\min(t,1)
\frac{dt}{t^{3/2}}=4,
\end{align*}
to obtain the statement of the present lemma.
\end{proof}

\section{The incidence bound}\label{incidencesection}
In this section we briefly explain how one can modify the proof of \cite[Theorem 1.2]{Zar} to obtain Theorem \ref{genipv}. As in \cite{Zar} the result will follow from a more general statement. First we recall the notions of degree and dimension of a real algebraic variety since they are key features of the proof of Theorem \ref{genipv}.\\
{\bf The degree and dimension of a real variety.} Let $V \subset \R^d$ be a real algebraic variety. Letting $I(V)$ denote the ideal of polynomials vanishing on $V$, we define $\dim(V)$ to be the Krull dimension of the quotient ring $\R^d/ I(V)$.\\
Let $V^{*}$ denote the complexification of $V$ (i.e. the Zariski closure of $V$, viewed as a subset of $\C^d$). As discussed in \cite[Section 4.1]{Zar}, the notion of degree is well-defined for complex varieties, so we may take $\deg(V)$ to be the degree of the complex variety $V^{*}$. One has the following relationship between the \textsl{complexity} and degree of a complex variety $W$ (see \cite[Lemmas 4.2 and 4.3]{ST}): 
\begin{itemize}
\item Any irreducible variety $W \subset \C^d$ of degree at most $D$ can be expressed as the zero set $\left\{ \underline{z} \in \C^d: g_i(\underline{z})=0 \ \forall i \leq r  \right\}$ where each polynomial $g_i \in \C[z_1,...,z_d]$ has degree at most $D$ and  $r=O_{d,D}(1)$. 
\item Suppose $W=\left\{ \underline{z} \in \C^d: g_i(\underline{z})=0 \ \forall i \leq r  \right\}$ is cut out by polynomials of degree at most $D$. Then $W$ can be decomposed into $O_{r,d,D}(1)$ irreducible varieties each having degree $O_{r,d,D}(1)$.
\end{itemize}
Given a polynomial $f$ we will denote by $Z(f)$ its zero set.

\begin{thm}[{\cite[Theorem 6.4]{Zar}}, quantitative in $s,t$]
\label{genipv2}
Let $\mathcal{P} \subset \R^d$ be a set of $k$ points and $\mathcal{V} $ a collection of $n$ algebraic varieties of bounded degree \footnote{Here we are assuming there is a constant $C$ such that each $V \in \mathcal{V}$ can be written as an algebraic set $\left\{ \underline{x} \in \R^d: p_i(\underline{x})=0 \ \forall i \leq r  \right\}$ where $r \leq C$ and each $p_i \in \R[x_1,...,x_d]$ has degree at most $C$.    }  in $\R^d$. Suppose the incidence graph of $\mathcal{P}\times \mathcal{V} $ is $K_{s,t}$-free and that $\mathcal{P}$ is contained in some irreducible variety $X$ of dimension $e$ and degree $D$. Lastly, suppose that no variety $V \in \mathcal{V}$ contains $X$. Then for any $\varepsilon >0$ there are positive constants $c_1(e)=c_1(\varepsilon,d,D,e)$ and $c_2(e)=c_2(\varepsilon,d,D,e)$  so that
\begin{equation}\label{gentinc2}
I(\mathcal{P},\mathcal{V} )\leq s t \left(  c_1(e) k^{ \frac{s(e-1)}{es-1}+\varepsilon  } \cdot n^{\frac{e(s-1) }{es-1}  } + c_2(e)(k+n) \right).
\end{equation}
\end{thm}

Observe that Theorem \ref{genipv} follows immediately, taking $X=\R^3$ and $d=3$. The proof of Theorem \ref{genipv2} is carried out exactly as in \cite[Theorems 4.3 and 6.4]{Zar} except that one requires a quantitative version of the classical K\H{o}v\'ari-S\'os-Tur\'an Theorem. 

\begin{thm}[\cite{H}]\label{KST}
Let $G=(V_1,V_2,E)$ be a bipartite graph with $|V_1|=k$, $|V_2|=n$ and suppose that $G$ does not contain a copy of $K_{s,t}$. Then $|E|\leq (t-1)^{1/s} k n^{1-1/s} + (s-1)n$.
\end{thm}

The final ingredient is a result from \cite{ST}.
\begin{thm}\cite[Theorem A.2]{ST} \label{conncomps}
Let $V \subset \R^d$ be an irreducible variety of dimension $h$ and degree $D$ and let $f \in \R[x_1,...,x_d]$ be a polynomial of degree $M \geq 1$. Then $V \setminus Z(f)$ has at most $O_D( M^h)$ connected components.
\end{thm}

{\bf Sketch of the proof of Theorem \ref{genipv2}.} The inequality (\ref{gentinc2}) is established by means of a two-step induction argument on the quantities $e$ and $k+n$.

{\bf Base cases.} When $e=0$, the irreducible variety $X$ consists of a single point and hence \eqref{gentinc2} is easily satisfied. On the other hand, when $k+n$ is small (regardless of the dimension $e$) we can choose $c_1(e), c_2(e)$ to be sufficiently large, thereby satisfying \eqref{gentinc2}. 

Let $r$ be a large number to be determined later. By the polynomial partitioning method (in the modified form \cite[Theorem 4.2]{Zar}) there exists a polynomial $f \in \R[x_1,...,x_d] \setminus I(X)$ of degree at most $O_{d,D}(r^{1/e})$ so that each connected component of $\R^d \setminus Z(f)$ contains at most  $k/r$ points of $\mathcal{P}$.\\
Defining $X_f:=X \cap Z(f)$, we may now split the set of incidences $\mathcal{I} \subset \mathcal{P} \times \mathcal{V}$ into three parts (recall that $X$ is not contained in any $V \in \mathcal{V}$):
\begin{itemize}
\item[-] $\mathcal{I}_1$ is given by those $(p,V) \in \mathcal{I}$ for which $p \in Z(f)$ and $V$ properly intersects each irreducible component of $X_f$ that contains $p$.  
\item[-] $\mathcal{I}_2$ is given by those  $(p,V) \in \mathcal{I}$ for which $p$ lives in some irreducible component of $X_f$ that contains $V$.
\item[-] $\mathcal{I}_3$ consists of all remaining incidences, i.e. those  $(p,V) \in \mathcal{I}$ with $p \notin X_f$.
\end{itemize}  

Throughout the remainder of the argument we may assume that $n \leq k^s$. Indeed, when $k^s <n$ we get that $n^{-1/s} < k^{-1}$ and hence by Theorem \ref{KST}   
$$I(\mathcal{P},\mathcal{V} )\leq  t^{1/s} k n^{1-1/s} +sn \leq (t+s) n,$$
yielding the desired estimate (\ref{gentinc2}). We record the inequality
\begin{equation}\label{nmn}
n = n^{\frac{e-1}{es-1}}   n^{\frac{e(s-1) }{es-1}} \leq k^{\frac{s(e-1)}{es-1}}  n^{\frac{e(s-1) }{es-1}}.
\end{equation}

Let $k_f=|\mathcal{P} \cap X_f|$. Since $X$ is irreducible of dimension $e$ and $f \notin I(X)$ we have that $\dim(X_f)=:e'\leq e-1$. By \cite[Theorem 2]{royvor} (note that the results in \cite{royvor} are described in terms of both complexity and degree) and the discussion preceding Theorem \ref{genipv2} we can decompose $X_f$ into $l$ irreducible varieties, each of degree at most $l$ and dimension at most $e'$. Here $l$ depends on the quantities $d,D,e,r$. 
\newline {\bf A bound for $\mathcal{I}_1$.} Applying the induction hypothesis it follows that
\begin{equation}\label{i11}
|\mathcal{I}_1| \leq st l \left(  c_1(e-1) k^{ \frac{s(e-2)}{(e-1)s-1}+\varepsilon  } n^{\frac{(e-1)(s-1) }{(e-1)s-1}  } + c_2(e-1)(k_f+n) \right).
\end{equation} 
Invoking the estimate \cite[(12)]{Zar} one has
\begin{equation}\label{algmanip}
k^{ \frac{s(e-2)}{(e-1)s-1}+\varepsilon  } n^{\frac{(e-1)(s-1) }{(e-1)s-1}  } \leq k^{ \frac{s(e-1)}{es-1}+\varepsilon  } n^{\frac{e(s-1) }{es-1}  }.
\end{equation}  
Choosing $c_1(e) \geq 3 l c_1(e-1), c_2(e) \geq 3 l c_2(e-1)$ and inserting \eqref{algmanip} into \eqref{i11} we get
\begin{equation}\label{i1f}
|\mathcal{I}_1| \leq \frac{st}{3} \left(  c_1(e) k^{ \frac{s(e-1)}{es-1}+\varepsilon  } n^{\frac{e(s-1) }{es-1}  } + c_2(e)(k_f+n) \right).
\end{equation}

{\bf A bound for $\mathcal{I}_2$.} Since the incidence graph of $\mathcal{P} \times \mathcal{V}$ is $K_{s,t}$-free, each irreducible component of $X_f$ either contains at most $s-1$ points from $\mathcal{P}$ or is contained in at most $t-1$ varieties belonging to $\mathcal{V}$. As a result $|\mathcal{I}_2| \leq l s n+ t k_f$. Since we have already assumed $c_2(e) \geq 3l$, it follows that
\begin{equation}\label{i2f}
|\mathcal{I}_2| \leq \frac{st}{3} c_2(e)( k_f + n ).
\end{equation}

{\bf A bound for $\mathcal{I}_3$.} Let $k'_f:=k-k_f$. By Theorem \ref{conncomps} the set $\R^d \setminus Z(f)$ can be partitioned into connected components (or \textsl{cells}) $\Omega_1,...,\Omega_J$ with $J =O_{D,d}( r) $. Given any $(p,V) \in \mathcal{I}_3$ we first note that the variety $V$ must properly intersect one of the cells $\Omega_j$.\\
Among the $O(1)$ polynomials defining $V$ there must be at least one, say $g$, for which $Z(g)$ does not fully contain $X$. Since $X$ is irreducible we have that $\dim(Z(g) \cap X ) \leq e-1$ and hence, by Theorem \ref{conncomps} (and the discussion preceding Theorem \ref{genipv2}), the variety $Z(g) \cap X$ intersects at most $O_{D,d}(r^{(e-1)/e  })$ cells. As a consequence $V \cap X$ intersects at most $O_{D,d}(r^{(e-1)/e  })$ cells. Introduce for each $j=1,...,J$
$$\mathcal{V}_{X}^{(j)}=\left\{V\cap X |\ V \in \mathcal{V} \text{ and } V \text{ intersects } \Omega_j  \right\}, \qquad \mathcal{P}^{(j)}=\mathcal{P} \cap \Omega_j.$$
Noting that 
$$|\mathcal{I}_3| =\sum_{j=1}^{J} I\left(\mathcal{P}^{(j)}, \mathcal{V}_{X}^{(j)} \right),$$
the  argument proceeds in precisely the same manner as \cite[Theorem 4.3]{Zar} and one gets
$$\sum_{j=1}^{J} I\left(\mathcal{P}^{(j)}, \mathcal{V}_{X}^{(j)} \right) \leq st  \left(  c_1(e) \kappa_1 r^{- \varepsilon}  k^{ \frac{s(e-1)}{es-1}+\varepsilon}  n^{\frac{e(s-1) }{es-1} }   + c_2(e)(k'_f+ \kappa_2 r^{(e-1)/e }n) \right),$$  

\noindent where $\kappa_1, \kappa_2$ depend only on the parameters $D,d,e$. Taking $\kappa_2 c_2(e)r^{(e-1)/e } \leq c_1(e)/4$ and choosing $r$ to be sufficiently large with respect to $\varepsilon$ and $\kappa_1$, it follows from \eqref{nmn} that 
\begin{equation}\label{i3f}
|\mathcal{I}_3| =\sum_{j=1}^{J} I \left(\mathcal{P}^{(j)}, \mathcal{V}_{X}^{(j)} \right) \leq  st   \left(  \frac{c_1(e) }{3}  k^{ \frac{s(e-1)}{es-1}+\varepsilon}  n^{\frac{e(s-1) }{es-1} }   + c_2(e)k'_f \right).\end{equation}
It remains to collect the three estimates \eqref{i1f}, \eqref{i2f} and \eqref{i3f}.






\bibliographystyle{plain}
\bibliography{bibfile}

\end{document}